\documentclass[12pt]{amsart}
\usepackage[active]{srcltx}
\usepackage{verbatim}
\usepackage{epsfig,graphicx,color,mathrsfs}
\usepackage{graphicx}
\usepackage{amsmath,amssymb,amsthm,amsfonts}
\usepackage{amssymb}
\usepackage[english]{babel}

\usepackage[left=2.7cm,right=2.7cm,top=3cm,bottom=3cm]{geometry}

\newcommand{\R}{{\mathbb R}}

\newcommand{\tu}{\tilde{u}}
\newcommand{\blambda }{\bar{\lambda}}

\newcommand{\vep}{\varepsilon}
\newcommand{\eps}{\varepsilon}

\numberwithin{equation}{section}
\newtheorem{theorem}{Theorem}[section]
\newtheorem{proposition}[theorem]{Proposition}
\newtheorem{lemma}[theorem]{Lemma}
\newtheorem{corollary}[theorem]{Corollary}

\newtheorem{remark}[theorem]{Remark}
\theoremstyle{definition}

\newcommand{\brm}{\begin{remark}\rm}
\newcommand{\erm}{\end{remark}}
\newcommand{\brms}{\begin{remark}\rm}
\newcommand{\erms}{\end{remark}}
\newcommand{\bte}{\begin{theorem}}
\newcommand{\ete}{\end{theorem}}
\newcommand{\bpr}{\begin{proposition}}
\newcommand{\epr}{\end{proposition}}
\newcommand{\ble}{\begin{lemma}}
\newcommand{\ele}{\end{lemma}}
\newcommand{\beq}{\begin{equation}}
\newcommand{\eeq}{\end{equation}}
\newcommand{\bdm}{\begin{displaymath}}
\newcommand{\edm}{\end{displaymath}}
\numberwithin{equation}{section}

\newcommand{\bos}{\begin{remark}\rm}
\newcommand{\eos}{\end{remark}}

\newcommand{\ben}{\begin{enumerate}}
\newcommand{\een}{\end{enumerate}}

\newcommand{\e }{\varepsilon }

\newcommand{\n }{\nabla }

\renewcommand{\t }{\tau }

\newcommand{\be}{\begin{equation}}
\newcommand{\ee}{\end{equation}}

\title[Monotonicity of solutions to some quasilinear problems]
{Monotonicity of solutions to  quasilinear problems with  a first-order 
term in half-spaces}

\author[A.\ Farina]{Alberto Farina$^+$}
\address{Universit\'e de Picardie Jules Verne
\newline\indent
LAMFA, CNRS UMR 6140\newline\indent
Amiens, France}
\email{alberto.farina@u-picardie.fr}

\author[L.\ Montoro]{Luigi Montoro$^*$}
\address{Dipartimento di Matematica
\newline\indent
Universit\`a della Calabria
\newline\indent
Ponte Pietro Bucci 31B, I-87036 Arcavacata di Rende, Cosenza, Italy}
\email{montoro@mat.unical.it}
\author[G.\ Riey]{Giuseppe Riey$^*$}
\address{Dipartimento di Matematica
\newline\indent
Universit\`a della Calabria
\newline\indent
Ponte Pietro Bucci 31B, I-87036 Arcavacata di Rende, Cosenza, Italy}
\email{riey@mat.unical.it}
\author[B.\ Sciunzi]{Berardino Sciunzi$^*$}
\address{Dipartimento di Matematica
\newline\indent
Universit\`a della Calabria
\newline\indent
Ponte Pietro Bucci 31B, I-87036 Arcavacata di Rende, Cosenza, Italy}
\email{sciunzi@mat.unical.it}
\thanks{\it 2000 Mathematics Subject
Classification: 35B05,35B65,35J70}
\thanks{$^+$Universit\'e de Picardie J.Verne,
LAMFA, CNRS UMR 7352,
Amiens, France,
E-mail: {\em alberto.farina@u-picardie.fr}}
\thanks{$^+$ A.F. is supported by the ERC grant EPSILON ({\it Elliptic Pde's and Symmetry of Interfaces and Layers for Odd Nonlinearities)}}

\thanks{$^*$Dipartimento di Matematica,
Universit\`a della Calabria,
Ponte Pietro Bucci 31B, I-87036 Arcavacata di Rende, Cosenza, Italy,
E-mail: {\em montoro@mat.unical.it}, {\em riey@mat.unical.it}, {\em sciunzi@mat.unical.it}}
\thanks{$^*$The authors were partially supported by the
Italian PRIN Research Project 2007: {\em Metodi Variazionali e Topologici nello Studio di Fenomeni non Lineari}}
\begin{document}
\begin{abstract}
We consider a quasilinear elliptic equation involving a first order term, under zero Dirichlet boundary condition in half spaces. We prove that any positive solution is monotone increasing w.r.t. the direction orthogonal to the boundary. The main ingredient in the proof is a new comparison principle in unbounded domains. As a consequence of our analysis, we also obtain some new Liouville type theorems.
\end{abstract}

\maketitle
\tableofcontents

\medskip

\section{Introduction and statement of the  main results.}\label{introdue}

We consider $C^{1,\alpha}$ weak solutions to the problem
\begin{equation}\label{E:P}
\begin{cases}
-{\hbox {\rm div}} (a(u)|\nabla u|^{p-2} \nabla u )+b(u)|\nabla u|^q=f(u), & \text{ in }\mathbb{R}^N_+\\
u(x',y) > 0, & \text{ in } \mathbb{R}^N_+\\
u(x',0)=0,&  \text{ on }\partial\mathbb{R}^N_+
\end{cases}
\end{equation}
where we assume $N\ge2$, $\alpha \in (0,1)$ and
\begin{itemize}
\item [($H_1$)] $1<p<2$,  $1<q\leq p$;
\item [($H_2$)] $a,b$ and $f$ are locally Lipschitz continuous functions on ${\mathbb{R}}$;
\item [($H_3$)] there exists $\gamma>0$ such that $a(s)\ge\gamma$ for every $s \in {\mathbb{R}}$.
\end{itemize}
We denote a generic point in
$\mathbb{R}^N_+$ by $x=(x',y)$ with $x'=(x_1,x_2, \ldots, x_{N-1})$ and $y=x_N$.\\

Note that the $C^{1,\alpha}$ regularity of the solutions follows by the well known results in \cite{Di, LU68, Li, T}.

We study monotonicity properties of the solutions, w.r.t. the $y$-direction, via the Alexandrov-Serrin moving plane method \cite{A,BN,GNN,S}. For the semilinear case, the founding  papers on this topic go back to  the works
\cite{BeCaNi1, BCN1,BCN2,BCN3, Dan1, GiSp}.  
We also refer the readers to \cite{CMS,DaGl,Dan2,DuGuo,Fa,FV2,QS}  for other results concerning the monotonicity of the solutions in half-spaces also in more general settings (always in the uniformly elliptic case). \

In the present work, we consider  the quasilinear problem \eqref{E:P} and we continue the study that we have started in \cite{FMS,FMS3}. \\

The first main contribution of this paper is the following:
\begin{theorem}\label{mainthm}
Let  $u$  be a solution  to \eqref{E:P} and let us assume that $u \in C^{1,\alpha}_{loc}(\overline{\mathbb{R}^N_+})$ and $\n u \in L^{\infty}(\mathbb{R}^N_+)$.  Let $(H_1)$, $(H_2)$ and $(H_3)$ be satisfied and assume that $f(s)>0$ for $s>0$.\\
 \noindent  Then $u$ is monotone increasing w.r.t. the $y$-direction, that is
 \[
\frac{\partial u}{\partial y}\,\geq\,0 \quad \text{in}\quad \mathbb{R}^N_+\,.
 \]
\end{theorem}

For $a(\cdot)=1$ and $b(\cdot)=0$ problem \eqref{E:P} reduces to
\begin{equation}\label{E:Pbisygfjsdg}
\begin{cases}
-\Delta_p\,u=f(u), & \text{ in }\mathbb{R}^N_+\\
u(x',y) > 0, & \text{ in } \mathbb{R}^N_+\\
u(x',0)=0,&  \text{ on }\partial\mathbb{R}^N_+\,.
\end{cases}
\end{equation}

The monotonicity of solutions to \eqref{E:Pbisygfjsdg} was first studied in \cite{DS3} in the two dimensional case and considering positive nonlinearities. Later, in dimension $N$ (always for positive nonlinearities) a first result was obtained in \cite{FMS}. Both the results hold  under the restriction $\frac{2N+2}{N+2}<p<2$.\\

As a consequence of Theorem \ref{mainthm}, we can remove this restriction and get the following:\

 \begin{corollary}\label{mainthmdgbdioug}
Let $1<p< 2$ and $u \in C^{1,\alpha}_{loc}({\overline {{\mathbb{R}^N_+}}})$  be a solution of problem \eqref{E:Pbisygfjsdg} with  $|\n u| \in L^{\infty}(\mathbb{R}^N_+)$. Assume that $f$ is locally Lipschitz continuous with $f(s)>0$ for $s>0$.  Then $u$ is monotone increasing w.r.t. the $y$-direction.
\end{corollary}
Let us point out that in \cite{FMS}  the restriction $\frac{2N+2}{N+2}<p<2$ is needed because it is used there the strong maximum and comparison principles of \cite{DS2}, which, in turn,  are based on the estimates in \cite{DS1}. A novelty in this paper, even in the special case of the pure $p$-Laplacian operator, is the fact that we avoid the restriction $\frac{2N+2}{N+2}<p<2$.

Let us also mention that the case $p=2$ is well known, as remarked here above, while recently in \cite{FMS3} the monotonicity of solutions to \eqref{E:Pbisygfjsdg} is proved in the case $2< p<3$, for  positive power-like nonlinearities or in the case $p>2$, for strictly positive nonlinearities.\\

The technique exploited in the proof of Theorem \ref{mainthm} also allows us to improve Theorem~1.8 in \cite{FMS} allowing the presence of a first order term in the equation and avoiding also in this case the restriction $p>\frac{2N+2}{N+2}$. Namely we have the following:

\begin{theorem}\label{mainthmnonpositive}
Let $(H_1)$ and $(H_2)$  be satisfied.  Assume that $u \in C^{1,\alpha}_{loc}(\overline{\mathbb{R}^N_+})$, with $ u \in W^{1,\infty}(\mathbb{R}^N_+)$,  is a solution to \eqref{E:P} with $a\equiv1$ and $b(u)\geq 0$.  Suppose that
$$ \exists \, z > 0 \quad : \quad 0 < s < z \quad \Rightarrow \quad f(s)
> 0 \qquad \text{and }\qquad  s > z \quad \Rightarrow \quad f(s) <0.$$
 Then $u$ is monotone increasing w.r.t. the $y$-direction.
\end{theorem}

Let us emphasize some new Liouville type results that complement and improve those we proved in \cite[Theorem 1.6]{FMS}. They follow from our monotonicity results and some techniques used in \cite {FMS}:

\begin{theorem}\label{liouvillenextgenerationtris}
Assume $1<p<2$ and let $f$ be locally Lipschitz continuous. Let $u\in C^{1,\alpha}_{loc}(\overline{\R^N_+})\cap W^{1,\infty}(\R^N_+)$ be a non-negative solution of

\begin{equation}\label{E:nonN}
\begin{cases}
-\Delta_p\,u =f(u), & \text{ in }\mathbb{R}^N_+\\
u(x',0)=0,&  \text{ on }\partial\mathbb{R}^N_+\,.
\end{cases}
\end{equation}


Assume that one of the following holds:

\medskip

\begin{itemize}
\item[a)] $N=2$ and  $f(s)>0$ for $s>0$, with $f(0) =0$,

  \

\item[b)] $N\geqslant 3$,
$f(s)>0$ for $s>0$, $f(0) =0$ and  $f$ is subcritical w.r.t. the Sobolev critical exponent in $\mathbb{R}^{N-1}$,

\

\item[c)] $N\geqslant 3$,
$f(s)>0$ for $s>0$, $f(0) =0$ and  $f(s)\geq \lambda s^\frac{(N-1)(p-1)}{N-1-p}$ in $[0,\delta]$, for some $\lambda,\delta>0$.
\end{itemize}

\smallskip

\noindent Then $u\equiv0$. \

On the other hand, if $ N\geqslant 2$, $f(s)>0$ for $s\ge 0$, then there are no non-negative solutions of \eqref{E:nonN}.

\end{theorem}
We refer the readers to \cite{DS3, FMS,FMS3,Zou} for other Liouville type theorems for quasilinear elliptic equations in half-spaces. \\

The proofs of the monotonicity results in half spaces are generally based on weak comparison principles in narrow unbounded domains. We refer the readers to  \cite{BCN1,BCN2,BCN3,DaGl,Dan1,Dan2,Fa,FV2,QS}.

In our case, the presence of the therm $|\nabla u|^{p-2}$ gives rise to a phenomenon that was first pointed out in \cite{lucio,DP}, in the case of bounded domains. Namely, we will prove our monotonicity result via a weak comparison principle
in  domains that can be decomposed into two parts. A narrow part (w.r.t. the Lebesgue measure of the section) and a part where the gradient of the solution is small.

We will state our comparison principle, which is the second main contribution of the present work, under more general structural assumptions and for more general quasilinear problems (cfr.\eqref{Eq:WCP} below). We also remark that no restrcition on the sign of $u$ and $v$ is required. More precisely, we consider continuous functions $a=a(x,u),b=b(x,u)$ and $f=f(x,u)$ defined on ${\mathbb{R}}^N_+ \times {\mathbb{R}}$ and satisfying the following conditions:
\begin{itemize}
\item [($H_4$)] $a, b $ and $f$ are locally Lipschitz continuous functions, uniformly w.r.t. $x$. Namely, for every $M>0$,  there are positive constants $L_a=L_a(M)\,\,,\,\,L_b=L_b(M)$ and $L_f=L_f(M)$ such that for every $x \in {\mathbb{R}}^N_+$ and every $ u,v \in [-M,M] $ we have :
$$ \vert a(x,u) - a(x,v) \vert \le L_a \vert u-v \vert, \qquad \vert b(x,u) - b(x,v) \vert \le L_b \vert u-v \vert, $$
$$ \vert f(x,u) - f(x,v) \vert \le L_f \vert u-v \vert. $$
For every $M>0$ there is a constant $K=K(M)>0$ such that for every $x \in {\mathbb{R}}^N_+$ and every $ s \in [-M,M] $ we have :

$$ \vert a(x,s) \vert \le K, \qquad \vert b(x,s) \vert \le K. $$

\item [($H_5$)]  There is a constant $\gamma > 0 $ such that $a(x,s)\geq\gamma$ for every $(x,s) \in {\mathbb{R}^N_+} \times \mathbb{R}$.
\end{itemize}

\

\begin{remark} (i) When the functions $a$, $b$ and $f$ depend only on the variable $u$ and are locally Lipschitz continuous on $ \mathbb{R}$, the assumption $(H_4)$ is automatically satisfied.

\smallskip

\noindent (ii) Typical examples of functions $a=a(x,u)$ satisfying both $(H_4)$ and $(H_5)$ are provided by $ a(x,u) = (\sin^2(x_1) +10) (\vert u \vert +1)$ or $ a(x,u) = \vert u \vert +2 + \cos^2(\vert x \vert)$.
\end{remark}

\medskip

We have the following:

\begin{theorem}\label{th:wcpstrip}
Let $1<p<2$, $N \geq 2$ and let $(H_1),(H_4),(H_5)$ be satisfied. Fix $\lambda_0 >0$ and $M_0>0$. Consider $\lambda \in (0, \lambda_0]$, $\tau,\varepsilon>0$ and set
\begin{equation}\label{eq:gadkgdfakhsjgfasfgjhksagfjfgjasgjh}
\Sigma_{(\lambda, y_0)}:= \mathbb{R}^{N-1}\times \big(y_0-\frac{\lambda}{2}, y_0+\frac{\lambda}{2}\big), \quad y_0 \geq \frac{\lambda}{2}.
\end{equation}
Let $u, v \in C^{1, \alpha}_{loc}(\overline{\Sigma_{(\lambda, y_0)}})$ such that $ \Vert u \Vert_{\infty} + \Vert \nabla u \Vert_{\infty} \le M_0$, $ \Vert v \Vert_{\infty} + \Vert \nabla v \Vert_{\infty} \le M_0$ and
\begin{equation}\label{Eq:WCP}
\begin{cases}
-{\hbox {\rm div}} (a(x,u)|\nabla u|^{p-2} \nabla u )+b(x,u)|\nabla u|^q\leq f(x,u), & \text{ in }\Sigma_{(\lambda, y_0)},\\
-{\hbox {\rm div}} (a(x,v)|\nabla v|^{p-2} \nabla v )+b(x,v)|\nabla v|^q\geq f(x,v), & \text{ in }\Sigma_{(\lambda, y_0)},\\
u\leq v, & \text{ on } \partial\mathcal{S}_{(\t,\varepsilon)}\,,
\end{cases}
 \end{equation}
where  the open set $\mathcal{S}_{(\t,\varepsilon)}\subseteq \Sigma_{(\lambda, y_0)}$ is such that
$$
\mathcal{S}_{(\t,\varepsilon)}=
 \bigcup_{x'\in\mathbb R^{N-1}} I_{x'}^{\t,\eps}\,,
$$
and  the open set $I_{x'}^{\t,\eps}\subseteq \{x'\}\times(y_0-\frac{\lambda}{2}, y_0+\frac{\lambda}{2})$ has the form
\begin{equation}\label{split}
I_{x'}^{\t,\eps}=A_{x'}^\t\cup B_{x'}^\eps\qquad \text{with}\quad |A_{x'}^\t\cap B_{x'}^\eps|=0
\end{equation}
and, for $x'$ fixed,   $A_{x'}^\t\,,\,B_{x'}^\varepsilon\subset(y_0-\frac{\lambda}{2}, y_0+\frac{\lambda}{2})$ are measurable  sets such that
\begin{center}
$|A_{x'}^\t|\le \t \quad$ and
$\quad B_{x'}^\eps \subseteq\left\{y\in\R:\, |\nabla u(x',y)|<\eps,\,|\nabla v(x',y)|<\eps\right\}$.
\end{center}
Then there exist $$\tau_0=\tau_0(N,p,q,\lambda_0, M_0, \gamma)>0$$
and
$$\eps_0=\eps_0(N,p,q,\lambda_0, M_0, \gamma)>0 \footnote{Both $\eps_0$ and $\tau_0$ can be explicitely calculated, cfr. Remark \ref{valoriexpl}.}$$
such that, if  $0<\tau<\tau_0$ and $0<\eps<\eps_0$, it follows that
$$
u \leq v \qquad \text{ in } \quad \mathcal{S}_{(\t,\varepsilon)}.
$$
If the functions $f$, $a$ and $b$ are assumed to be \emph{globally} Lipschitz continuous on $ {\mathbb{R}^N_+} \times {\mathbb{R}}$ and the functions $a$ and $b$ are supposed to be \emph{bounded} on $ {\mathbb{R}^N_+} \times {\mathbb{R}}$, the same conclusion holds true without any assumption on the boundedness of $u$ and $v$.
\end{theorem}

For later purposes, we also state the following special case of the previous theorem. It corresponds to the case in which $B_{x'}^\varepsilon\equiv\emptyset$ and the set $\mathcal{S}_{(\t,\varepsilon)}$ is contained in a narrow strip.
This result also provides an extension of Theorem 1.1 in \cite{FMS} to the case of problems involving a
first-order term as in \eqref{E:P} or in \eqref{Eq:WCP}.

\begin{theorem}\label{th:COROwcpstrip}
Let $1<p<2$, $N \geq 2$ and let $(H_1),(H_4),(H_5)$ be satisfied. Fix $M_0>0$. Consider $\lambda>0$ and  set
$$
\Sigma_{(\lambda, y_0)}:= \mathbb{R}^{N-1}\times \big(y_0-\frac{\lambda}{2}, y_0+\frac{\lambda}{2}\big), \quad y_0 \geq \frac{\lambda}{2}.$$
Let $u, v \in C^{1, \alpha}_{loc}(\overline{\Sigma_{(\lambda, y_0)}})$
such that $ \Vert u \Vert_{\infty} + \Vert \nabla u \Vert_{\infty} \le M_0$, $ \Vert v \Vert_{\infty} + \Vert \nabla v \Vert_{\infty} \le M_0$ and
\begin{equation}\label{Eq:CoroWCP}
\begin{cases}
-{\hbox {\rm div}} (a(x,u)|\nabla u|^{p-2} \nabla u )+b(x,u)|\nabla u|^q\leq f(x,u), & \text{ in }\Sigma_{(\lambda, y_0)},\\
-{\hbox {\rm div}} (a(x,v)|\nabla v|^{p-2} \nabla v )+b(x,v)|\nabla v|^q\geq f(x,v), & \text{ in }\Sigma_{(\lambda, y_0)},\\
u\leq v, & \text{ on } \partial\mathcal{S}\,,
\end{cases}
 \end{equation}
where  $\mathcal{S} \subseteq \Sigma_{(\lambda, y_0)}$ is an open subset.

Then there exists $${\overline \lambda}={\overline \lambda}(N, p, q, M_0, \gamma)>0\footnote{ \, ${\overline \lambda} $ can be explicitely calculated, cfr. Remark \ref{valoriexpl}.}$$

such that, if  $0<\lambda<{\overline \lambda}$, it follows that
$$
u \leq v \qquad \text{ in } \, \mathcal{S}.
$$
If the functions $f$, $a$ and $b$ are assumed to be \emph{globally} Lipschitz continuous on $ {\mathbb{R}^N_+} \times {\mathbb{R}}$ and the functions $a$ and $b$ are supposed to be \emph{bounded} on $ {\mathbb{R}^N_+} \times {\mathbb{R}}$, the same conclusion holds true without any assumption on the boundedness of $u$ and $v$.
\end{theorem}

\begin{remark}
Theorem \ref{th:wcpstrip} and Theorem \ref{th:COROwcpstrip} are the main ingredients in the proof of our monotonicity result Theorem \ref{mainthm}. The new geometry of the domain that we consider is crucial  to prove our monotonicity result in the case $1<p< 2$, without the restriction $p>\frac{2N+2}{N+2}$ that appears in \cite{FMS}.  At each $x'$ fixed, the section of the domain is decomposed into two parts: a narrow (w.r.t. the Lebesgue measure) part and a part where the gradients are small.
\end{remark}

The qualitative properties of positive solutions to $-\Delta_p\,u=f(u)$, in various unbounded domains, are also studied in \cite{DFSV, DS3, efendiev, DuGuo, FMS, FMS3, FSV, FSV2, galakhov, MP, SZ, Zou}. \\

Here below we describe the scheme of the paper:

\begin{itemize}
\item [i)] In Section \ref{kjgkgòhgòs} we prove the general weak comparison principle stated in  Theorem \ref{th:wcpstrip} (and also Theorem \ref{th:COROwcpstrip}). The proof is carried out exploiting the iteration scheme introduced in \cite{FMS} (see also \cite{FMS3}), and taking advantage from the geometry of the considered domain.
    \item  [ii)] In Section \ref{sectioncompact} we prove a crucial property of local symmetry regions of the solutions. Namely we show that such regions must touch the boundary. This follows by a fine analysis of the limiting profiles of the solution.
        \item [iii)] In Section \ref{sectioncompactbistreremdbvfj} we prove Proposition \ref{lellalemma}, that allows to carry out  the moving plane procedure via Theorem \ref{th:wcpstrip}.
            \item [iv)] In Section \ref{sect.cort} we prove Theorem \ref {mainthm}.
            \item [v)] In Section \ref{sect.cortjlhjdligld} we prove Theorem \ref {mainthmnonpositive} and Theorem \ref {liouvillenextgenerationtris}.
\end{itemize}

\section{The weak comparison principle: Proof of Theorem \ref{th:wcpstrip}}\label{kjgkgòhgòs}

In the sequel we will use the following  inequalities: \noindent $\forall \eta, \eta' \in  \mathbb{R}^{N}$ with $|\eta|+|\eta'|>0$
there exist positive constants $C_1, C_2$ depending only on $p$ such that
\begin{eqnarray}\label{eq:inequalities}
[|\eta|^{p-2}\eta-|\eta'|^{p-2}\eta'][\eta- \eta'] &\geq& C_1 (|\eta|+|\eta'|)^{p-2}|\eta-\eta'|^2 \quad \text{if }    p >1  \\ \nonumber
\\ \nonumber
||\eta|^{p-2}\eta-|\eta'|^{p-2}\eta'| &\leq&C_2|\eta -\eta'|^{p-1} \quad \text{if } 1<p \le 2.\end{eqnarray}

\vskip6pt
\noindent
{\sc Notation.} Generic fixed numerical constants will be denoted by
$C$ (with subscript in some case) and they will be allowed to vary within a single line or formula.
$|S|$ denotes the Lebesgue measure of a set $S$.
$f^+$ and $f^-$ are the positive part and the negative part of a function $f$, i.e.
$f^+=\max\{f,0\}$ and $f^-=-\min\{f,0\}$.\\

We start  recalling a lemma, whose proof can be found in \cite[Lemma 2.1]{FMS}.
\begin{lemma}\label{Le:L(R)}
Let $\theta >0$ and $\beta>0$ such that $\theta < 2^{-\beta}$.
Moreover let $R_0>0$, $c>0$ and
$$\mathcal{L}:(R_0, + \infty) \rightarrow \mathbb{R}$$
a non-negative and non-decreasing function such that
\begin{equation}\label{eq:L}
\begin{cases}
\mathcal{L}(R)\leq \theta \mathcal{L}(2R)+g(R) & \forall R>R_0,\\
\mathcal{L}(R)\leq CR^{\beta} & \forall R >R_0,
\end{cases}
\end{equation}
where $g:(R_0, +\infty)\rightarrow \mathbb{R}^+$ is such that
$\displaystyle \lim_{R\rightarrow +\infty}g(R)=0 .$ Then
$$\mathcal{L}(R)=0.$$
\end{lemma}

We provide now the proof of a generalized version of the Poincar\'e  inequality in one dimension.
\begin{lemma}[Poincar\'e type inequality]\label{poincare 1D lemma}
Let $I$ be an open bounded subset of $\R$ and assume that $I=A\cup B$ with $|A\cap B|=0$, $A$ and $B$ measurable subsets of $I$.
Let $\rho:I\to\R\cup\{\infty\}$ be measurable and such that
\[
\underset{t\in I}{\inf}\,\rho(t)>0\,.
\]
Then for any $w\in H^1_0(I)$ such that $\int_I\rho(t)\vert \partial_tw\vert^2(t)dt$ is finite,
the following inequality holds:
\begin{equation}\label{poincare 1D ineq}
\int_I w^2(t)dt\leq
2|I|\max
\left\{
|A|\sup_{t\in A}\frac{1}{\rho(t)},\,|B|\sup_{t\in B}\frac{1}{\rho(t)}
\right\}
\int_I\rho(t)\vert \partial_tw\vert^2(t)dt.
\end{equation}
\end{lemma}
\proof

Since $w$ belongs to $H^1_0(I)$, there exists $a\in \overline I$ such that
$w(x)=\displaystyle\int_a^x \partial_t w(t)dt$. Thus we have:
\begin{eqnarray}\label{eq:cazcczczczczcczzcc}\\
\nonumber |w(x)| &\leq&\int_a^x |\partial_t w(t)|dt\leq \int_I|\partial_tw(t)|dt=\int_A|\partial_tw(t)|dt+\int_B|\partial_tw(t)|dt\\
\nonumber &\leq& {|A|^{1/2}}{\left(\int_A \vert \partial_tw\vert^2(t)dt\right)^{1/2}}+{|B|^{1/2}}{\left(\int_B \vert \partial_tw\vert^2(t)dt\right)^{1/2}}\\
\nonumber &\leq& \left({|A|\sup_{t\in A}\frac{1}{\rho(t)}}\right)^{1/2}\left({\int_A\rho(t) \vert \partial_tw\vert^2(t)dt}\right)^{1/2}\\\nonumber&+&
 \left({|B|\sup_{t\in B}\frac{1}{\rho(t)}}\right)^{1/2}\left({\int_B\rho(t) \vert \partial_tw\vert^2(t)dt}\right)^{1/2}.
\end{eqnarray}

Finally, by using \eqref{eq:cazcczczczczcczzcc}  we obtain:
\begin{eqnarray}\nonumber
&& \int_I w^2(t)dt\leq |I|\sup_{t\in I} w^2(t)\\
\nonumber &\leq& 2|I|\left({|A|\sup_{t\in A}\frac{1}{\rho(t)}}{\int_A\rho\, \vert \partial_tw\vert^2(t)dt}+
|B| {\sup_{t\in B}\frac{1}{\rho(t)}}{\int_B\rho\, \vert \partial_tw\vert^2(t)dt}\right),
\end{eqnarray}
from which the thesis immediately follows.
\endproof

\

\noindent \textbf{Proof of Theorem \ref{th:wcpstrip}:} In the proof, the quantities $L_a(M_0), L_b(M_0), L_f(M_0)$ and $ K(M_0)$ will denote the structural constants appearing in assumption $(H_4)$. Also we denote by
$||\cdot||_\infty$, the $L^\infty$ norm in $\Sigma_{(\lambda, y_0)}$. \\

We remark that  $(u-v)^+$ belongs to $L^{\infty}(\Sigma_{(\lambda, y_0)})$ since $u$ and $v$
are bounded in $\Sigma_{(\lambda, y_0)}$.\\
For $\alpha>1$ we define
\begin{equation}\label{Eq:Cut-off}
\psi=[(u-v)^+]^{\alpha} \varphi^2,
\end{equation}
where $\varphi(x',y)=\varphi(x')$ and  $\varphi(x')\in C^{\infty}_c (\mathbb{R}^{N-1}) $ is such that
\begin{equation}\label{Eq:Cut-off1}
\begin{cases}
\varphi \geq 0, & \text{ in } \mathbb{R}^N_+\\
\varphi \equiv 1, & \text{ in } B^{'}(0,R) \subset \mathbb{R}^{N-1},\\
\varphi \equiv 0, & \text{ in } \mathbb{R}^{N-1} \setminus B^{'}(0,2R),\\
|\nabla \varphi | \leq \frac CR, & \text{ in } B^{'}(0, 2R) \setminus B^{'}(0,R) \subset  \mathbb{R}^{N-1},
\end{cases}
\end{equation}
where $B^{'}(0,R)=\left\{ x'\in \mathbb{R}^{N-1}: |x'|<R \right\}$, $ R>1$ and $C$ is a positive constant.

Let $\mathcal{C}(R)$ be  defined as
$$\mathcal{C}(R):=
\left\{ \mathcal{S}_{(\tau,\varepsilon)}\cap \{B^{'}(0,R)\times \mathbb{R}\} \right\}.
$$

The assumptions in \eqref{Eq:Cut-off1} and the inequality
$u\leq v \text{ on } \partial \mathcal{S}_{(\tau,\varepsilon)}$
imply that $\psi \in W_0^{1,p}(\mathcal{C}(2R))$.
This allows us to use $\psi$  as test function in both equations of problem
\eqref{Eq:WCP} and to get (by subtracting):
\begin{eqnarray}\nonumber
&&\int_{\mathcal{C}(2R)}(a(x,u)|\n u|^{p-2}\n u-a(x,v)|\n v|^{p-2}\n v,\n\psi)\\&&+
\int_{\mathcal{C}(2R)}(b(x,u)|\n u|^q-b(x,v)|\n v|^q)\psi \\\nonumber
&&\leq\int_{\mathcal{C}(2R)}(f(x,u)-f(x,v))\psi\,,
\end{eqnarray}
from which we infer:
\begin{eqnarray}\nonumber
 && \int_{\mathcal{C}(2R)} a(x,u)(|\n u|^{p-2}\n u-|\n v|^{p-2}\n v,\n \psi)
 \\&&+\int_{\mathcal{C}(2R)} (a(x,u)-a(x,v))(|\n v|^{p-2}\n v,\n \psi) \\\nonumber &&+\int_{\mathcal{C}(2R)} b(x,u)(|\n u|^q-|\n v|^q)\psi+
\int_{\mathcal{C}(2R)}(b(x,u)-b(x,v))|\n v|^q\psi\\
\nonumber &&\leq \int_{\mathcal{C}(2R)}(f(x,u)-f(x,v))\psi
\end{eqnarray}
and hence:
\begin{eqnarray}\label{eq:strunzu}
&& \int_{\mathcal{C}(2R)} a(x,u)(|\n u|^{p-2}\n u-|\n v|^{p-2}\n v,\n \psi)\\
\nonumber &\leq&\int_{\mathcal{C}(2R)} |a(x,u)-a(x,v)||\n v|^{p-1}|\n \psi|+K(M_0) \int_{\mathcal{C}(2R)}||\n u|^q-|\n v|^q|\psi\\\nonumber&+&\int_{\mathcal{C}(2R)}|b(x,u)-b(x,v)||\n v|^q\psi+\int_{\mathcal{C}(2R)}|f(x,u)-f(x,v)|\psi\leq\\
\nonumber &\leq& L_a(M_0)\int_{\mathcal{C}(2R)}|\n v|^{p-1}|\n \psi||u-v|+ K(M_0)\int_{\mathcal{C}(2R)}||\n u|^q-|\n v|^q|\psi\\\nonumber &+& L_b(M_0)\int_{\mathcal{C}(2R)}|\n v|^q\psi|u-v|+ L_f(M_0)\int_{\mathcal{C}(2R)}\psi|u-v|.
\end{eqnarray}

Since $\nabla u$ and $\n v$ belongs to $L^\infty(\Sigma_{(\lambda, y_0)})$,
using \eqref{Eq:Cut-off} we obtain:
\begin{eqnarray}\label{eq:bcdbcwkgcweryott}
&&\alpha\int_{\mathcal{C}(2R)}a(x,u)(|\n u|^{p-2}\n u-|\n v|^{p-2}\n v,\n(u-v)^+)[(u-v)^+]^{\alpha-1}\varphi^2\\
\nonumber &+&\int_{\mathcal{C}(2R)} a(x,u)(|\n u|^{p-2}\n u-|\n v|^{p-2}\n v,\nabla \varphi^2)[(u-v)^+]^\alpha\\
\nonumber &\leq&
\alpha L_a(M_0) \|\nabla v\|_{\infty}^{p-1}\int_{\mathcal{C}(2R)} [(u-v)^+]^\alpha|\n (u-v)^+|\varphi^2\\\nonumber
&+& L_a(M_0)\|\nabla v\|_{\infty}^{p-1}\int_{\mathcal{C}(2R)}[(u-v)^+]^{\alpha+1}|\n \varphi^2|+\\
\nonumber &+& (L_b(M_0)\|\nabla v\|_{\infty}^{q}+L_f(M_0))\int_{\mathcal{C}(2R)} [(u-v)^+]^{\alpha+1}\varphi^2+\\
\nonumber &+& q K(M_0)  {M_0}^{q-1}\int_{\mathcal{C}(2R)}|\n(u-v)^+||[(u-v)^+]^\alpha\varphi^2,
\end{eqnarray}
where, in the last term  we used the mean value theorem and the boundedness of $\nabla u$ and $ \nabla v$ to deduce that:
$\Big||\nabla u|^q-|\nabla v|^q\Big|\leq q {M_0}^{q-1}|\nabla(u-v)|.$\\

\noindent Recalling \eqref{eq:inequalities}, from \eqref{eq:bcdbcwkgcweryott}  we obtain
\begin{eqnarray}\nonumber
&& \alpha \, C_1\,\gamma \int_{\mathcal{C}(2R)} (|\nabla u|+|\nabla v|)^{p-2} |\nabla (u-v)^+ |^2 [(u-v)^+]^{\alpha-1}\varphi^2\\\nonumber
&\leq& C_2 K(M_0)\int_{\mathcal{C}(2R)} |\n (u- v)^+|^{p-1}|\nabla\varphi^2|[(u-v)^+]^\alpha\\
\nonumber &+&\alpha\,L_a(M_0)\,\|\nabla v\|_{\infty}^{p-1}\int_{\mathcal{C}(2R)}|\n (u-v)^+| [(u-v)^+]^{\frac{\alpha-1+\alpha+1}{2}}\varphi^2\\
\nonumber&+& L_a(M_0)\,\|\nabla v\|_{\infty}^{p-1}\int_{\mathcal{C}(2R)}[(u-v)^+]^{\alpha+1}|\n \varphi^2|\\
\nonumber &+&(L_b(M_0)\|\nabla v\|_{\infty}^{q}+L_f(M_0))\int_{\mathcal{C}(2R)} [(u-v)^+]^{\alpha+1}\varphi^2\\
\nonumber &+&
q K(M_0)  {M_0}^{q-1}\int_{\mathcal{C}(2R)}|\n (u-v)^+|[(u-v)^+]^{\frac{\alpha-1+\alpha+1}{2}}\varphi^2.
\end{eqnarray}

It is easy to resume as follows:
\begin{eqnarray}\label{ineq con fi 1}
&&\alpha\,  C_1\,\gamma \int_{\mathcal{C}(2R)} (|\nabla u|+|\nabla v|)^{p-2} |\nabla (u-v)^+ |^2 [(u-v)^+]^{\alpha-1}\varphi^2\\
\nonumber&\leq&
( C_2 K(M_0) (2 M_0)^{p-1} + 2 M_0^{p} L_a(M_0))\int_{\mathcal{C}(2R)} |\nabla\varphi^2|[(u-v)^+]^\alpha\\\nonumber
&+&\alpha\cdot ( L_a(M_0) M_0^{p-1} + q K(M_0)  {M_0}^{q-1})\int_{\mathcal{C}(2R)}|\n (u-v)^+|[(u-v)^+]^{\frac{\alpha-1+\alpha+1}{2}}\varphi^2\\
\nonumber &+&(L_b(M_0) M_0^{q}+L_f(M_0))\int_{\mathcal{C}(2R)} [(u-v)^+]^{\alpha+1}\varphi^2.
\end{eqnarray}

Recalling that $p<2$ and using the weighted Young inequality $zy\leq \e z^2+\frac{y^2}{4\e} $, from
\begin{equation}\nonumber
\begin{split}
&\int_{\mathcal{C}(2R)}|\n (u-v)^+|[(u-v)^+]^{\frac{\alpha-1+\alpha+1}{2}}\varphi^2  \\
&=\int_{\mathcal{C}(2R) \cap \{\vert \nabla u \vert + \vert \nabla v\vert >0\}}|\n (u-v)^+|[(u-v)^+]^{\frac{\alpha-1}{2}}\varphi \cdot [(u-v)^+]^{\frac{\alpha+1}{2}}\varphi
\end{split}
\end{equation}
\noindent we deduce that, for every  $\delta>0$ it holds
\begin{eqnarray}\label{ineq con fi 2}
\\\nonumber
&&\alpha\, C_1\,\gamma\int_{\mathcal{C}(2R)} (|\nabla u|+|\nabla v|)^{p-2}|\nabla (u-v)^+ |^2 [(u-v)^+]^{\alpha-1}\varphi^2\\
\nonumber &\leq&( C_2 K(M_0) (2 M_0)^{p-1} + 2 M_0^{p} L_a(M_0))\int_{\mathcal{C}(2R)} |\nabla\varphi^2| [(u-v)^+]^\alpha
\\\nonumber &+&  \delta\cdot\alpha\cdot ( L_a(M_0) M_0^{p-1} + q K(M_0)  {M_0}^{q-1})\int_{\mathcal{C}(2R)}|\nabla (u-v)^+ |^2 [(u-v)^+]^{\alpha-1}\varphi^2
\\\nonumber&+&
 \Big(\alpha \frac{( L_a(M_0) M_0^{p-1} + q K(M_0)  {M_0}^{q-1})}{4\delta}  + (L_b(M_0)M_0^q +L_f(M_0))\Big)\int_{\mathcal{C}(2R)}[(u-v)^+]^{\alpha+1}\varphi^2.\\
 \nonumber &\leq&( C_2 K(M_0) (2 M_0)^{p-1} + 2 M_0^{p} L_a(M_0))\int_{\mathcal{C}(2R)} |\nabla\varphi^2| [(u-v)^+]^\alpha
\\\nonumber & + &\delta\alpha2^{2-p}( L_a(M_0) M_0 + q K(M_0)  {M_0}^{q-(p-1)})\int_{\mathcal{C}(2R)}(|\nabla u|+|\nabla v|)^{p-2}|\nabla (u-v)^+ |^2 [(u-v)^+]^{\alpha-1}\varphi^2
\\\nonumber&+&
 \Big(\alpha \frac{2^{2-p}\cdot ( L_a(M_0) M_0 + q K(M_0)  {M_0}^{q-(p-1)})}{4\delta}  + (L_b(M_0)M_0^q +L_f(M_0))\Big)\int_{\mathcal{C}(2R)}[(u-v)^+]^{\alpha+1}\varphi^2.
\end{eqnarray}
Here we are using the fact that $ q>1> p-1$, since $(H_1)$ holds.
\noindent

Take now

\[
\delta \,:=\, \frac{C_1\,\gamma}{2\cdot(2^{2-p} L_a(M_0) M_0 + 2^{2-p}q K(M_0)  {M_0}^{q-(p-1)})}\,.
\]

Consequently we have:

\begin{eqnarray}\label{eq:blabalkaslskjjlksall2}
\\\nonumber
&& \qquad \alpha \frac{C_1\,\gamma}{2} \int_{\mathcal{C}(2R)} (|\nabla u|+|\nabla v|)^{p-2}|\nabla (u-v)^+ |^2 [(u-v)^+]^{\alpha-1}\varphi^2\\
\nonumber &\leq & ( C_2 K(M_0) (2 M_0)^{p-1} + 2 M_0^{p} L_a(M_0))\int_{\mathcal{C}(2R)} |\nabla\varphi^2| [(u-v)^+]^\alpha\\\nonumber&+&
\alpha \frac{(2^{2-p} L_a(M_0) M_0 + 2^{2-p}q K(M_0)  {M_0}^{q-(p-1)})^2}{2 C_1 \gamma} \int_{\mathcal{C}(2R)}[(u-v)^+]^{\alpha+1}\varphi^2\\
\nonumber&+&
 (L_b(M_0) M_0^{q}+L_f(M_0))\int_{\mathcal{C}(2R)}[(u-v)^+]^{\alpha+1}\varphi^2.
\end{eqnarray}

Since $ \alpha >1$ we immediately get that

\begin{eqnarray}\label{eq:blabalkaslskjjlksall}\\\nonumber
&& \qquad \int_{\mathcal{C}(2R)} (|\nabla u|+|\nabla v|)^{p-2}|\nabla (u-v)^+ |^2 [(u-v)^+]^{\alpha-1}\varphi^2\\
\nonumber &\leq& \frac{2 (C_2 K(M_0) (2 M_0)^{p-1} + 2 M_0^{p} L_a(M_0))}{ C_1 \gamma} \int_{\mathcal{C}(2R)} |\nabla\varphi^2| [(u-v)^+]^\alpha\\\nonumber&+&
 \frac{(2^{2-p} L_a(M_0) M_0 + 2^{2-p}q K(M_0)  {M_0}^{q-(p-1)})^2}{(C_1 \gamma)^2}  \int_{\mathcal{C}(2R)}[(u-v)^+]^{\alpha+1}\varphi^2
\\\nonumber&+&
 \frac{2 (L_b(M_0) M_0^{q}+L_f(M_0))}{C_1\gamma}  \int_{\mathcal{C}(2R)}[(u-v)^+]^{\alpha+1}\varphi^2.
\end{eqnarray}

\noindent Let us define

\begin{equation}\label{costante1}
c_1=\frac{2 (C_2 K(M_0) (2 M_0)^{p-1} + 2 M_0^{p} L_a(M_0))}{ C_1 \gamma},
\end{equation}
\begin{equation}\label{costante2}
c_2 =  \frac{(2^{2-p} L_a(M_0) M_0 + 2^{2-p}q K(M_0)  {M_0}^{q-(p-1)})^2}{(C_1 \gamma)^2} + \frac{2 (L_b(M_0) M_0^{q}+L_f(M_0))}{C_1\gamma} ,
\end{equation}

$$I_1= \int_{\mathcal{C}(2R)} |\nabla\varphi^2| [(u-v)^+]^\alpha,\quad I_2=\int_{\mathcal{C}(2R)}[(u-v)^+]^{\alpha+1}\varphi^2$$
and note that both $c_1$ and $c_2$ depend only on $p,q,\gamma $ and $ M_0$, in particular they are independent of $\alpha>1$.
\noindent Thus, with the definitions above, we now rewrite  \eqref{eq:blabalkaslskjjlksall} as follows: for every $\alpha>1$,
\begin{eqnarray}\label{eq:jdladlsaaljkjlkjkljkldscquhdoufhwou}
&& \int_{\mathcal{C}(2R)}(|\nabla u|+|\nabla v|)^{p-2} |\nabla (u-v)^+ |^2 [(u-v)^+]^{\alpha-1}\varphi^2\\
\nonumber &\leq&c_1I_1+c_2I_2.
\end{eqnarray}

We also observe that

$$  \int_{\mathbb{R}^{N-1}} \Big( \int_{I_{x'}^{\t,\eps}}(|\nabla u|+|\nabla v|)^{p-2} |\nabla (u-v)^+ |^2 [(u-v)^+]^{\alpha-1}\Big ) dy \varphi^2(x') dx'= $$
$$ = \int_{\mathcal{C}(2R)}(|\nabla u|+|\nabla v|)^{p-2} |\nabla (u-v)^+ |^2 [(u-v)^+]^{\alpha-1}\varphi^2 < +\infty $$
since $\varphi$ depends only on $x'$ and the right-hand-side of \eqref{eq:jdladlsaaljkjlkjkljkldscquhdoufhwou} is finite. Hence, for almost every $x' \in {\mathbb{R}^{N-1}} $ we have that

\begin{eqnarray}\label{eq:intPoincar}
&& \int_{I_{x'}^{\t,\eps}}(|\nabla u|+|\nabla v|)^{p-2} |\nabla (u-v)^+ |^2 [(u-v)^+]^{\alpha-1}dy < +\infty,
\end{eqnarray}
which also entails: for almost every $x' \in {\mathbb{R}^{N-1}} $

\begin{eqnarray}\label{eq:IntPoincar}
&& \int_{I_{x'}^{\t,\eps}}(|\nabla u|+|\nabla v|)^{p-2} |\partial_y (u-v)^+ |^2 [(u-v)^+]^{\alpha-1}dy < +\infty.
\end{eqnarray}

\medskip

{\bf $\hookrightarrow$ Evaluation of the term $I_1$}.
 Recalling the decomposition stated in \eqref{split} which gives
\begin{equation}\nonumber
\mathcal{S}_{(\t,\varepsilon)}=
 \bigcup_{x'\in\mathbb R^{n-1}} I_{x'}^{\t,\eps}\qquad\text{with}\qquad I_{x'}^{\t,\eps}=A_{x'}^\t\cup B_{x'}^\eps
\end{equation}
we set
$$
\rho_{x'}(t)=(|\nabla u(x',t)|+|\nabla v(x',t)|)^{p-2}
$$
in order to apply Lemma \ref{poincare 1D lemma} in each $I_{x'}^{\t,\eps}$, for which \eqref{eq:IntPoincar} holds true,  with $\rho(t):=\rho_{x'}(t)$, $A:=A_{x'}^\t$, $B:=B_{x'}^\eps$ and $ w(t) = [(u-v)^+(x',t)]^{\frac{\alpha+1}{2}}$.  Note that the constant in \eqref{poincare 1D ineq} in this case is given by:
$$
C_{\t,\varepsilon}(x')=2\lambda\,\max \left\{
|A_{x'}^\t|\sup_{t\in A_{x'}^\t}\frac{1}{\rho_{x'}(t)},\,|B_{x'}^\eps|\sup_{t\in B_{x'}^\eps}\frac{1}{\rho_{x'}(t)}
\right\}\,.
$$

Therefore, for almost every  $x'\in\mathbb{R}^{N-1}$, we have
\begin{equation}\label{fkjngòsfnbncbvcnbfv}
C_{\t,\varepsilon}(x')\leq C_{\t,\varepsilon}\,:=\, 2\lambda_0\,\max \left\{
\tau (2 M_0)^{2-p}\,,\,\lambda_0 (2\varepsilon)^{2-p}
\right\}\,,
\end{equation}
so that, since $1<p<2$, $C_{\t,\varepsilon}$ can be chosen arbitrary small, for $\t$ and $\varepsilon$ sufficiently small.
Now, recalling that $\varphi$ depends only on $x'$ and
using Young inequality with conjugate  exponents $\frac{\alpha+1}{\alpha}$ and $\alpha+1$,
we get:

$$
I_1 \leq 2 \int_{\mathcal{C}(2R)}[(u-v)^+]^{\alpha}\varphi|\nabla \varphi|=
2 \int_{\mathcal{C}(2R)}[(u-v)^+]^{\alpha}\varphi|\nabla \varphi|^{\frac 12}|\nabla \varphi|^{\frac 12} $$
$$ \leq 2\int_{\mathcal{C}(2R)}
  \frac{[(u-v)^+]^{\alpha+1}\varphi^{\frac{\alpha +1}{\alpha}}|\nabla \varphi|^{\frac{\alpha+1}{2\alpha}}}{\frac{\alpha +1}{\alpha}}+
  2 \int_{\mathcal{C}(2R)}\frac{|\nabla \varphi |^{\frac{\alpha+1}{2}}}{\alpha+1}$$
$$ \leq
  2\int_{\mathbb{R}^{N-1}}\left(\int_{I_{x'}^{\t,\eps}}\left([(u-v)^+]^{\frac{\alpha+1}{2}}\right)^2dy \right)
  \varphi^{\frac{\alpha +1}{\alpha}}|\nabla \varphi|^{\frac{\alpha +1}{2\alpha}}dx'+
  2\int_{\mathcal{C}(2R)}|\nabla \varphi |^{\frac{\alpha+1}{2}}
$$
and the application of Lemma \ref{poincare 1D lemma} yields

\begin{eqnarray} \label{Eq:I1}\\\nonumber
&I_1 \leq&
C_{\t,\eps}\frac{(\alpha+1)^2}{2}
  \int_{\mathcal{C}(2R)}
  (|\nabla u|+|\nabla v|)^{p-2}[(u-v)^+]^{\alpha-1}|\partial_y (u-v)^+|^2\varphi^{\frac{\alpha +1}{\alpha}}|\nabla \varphi|^{\frac{\alpha +1}{2\alpha}}\\
\nonumber &+& 2\int_{\mathcal{C}(2R)}|\nabla \varphi |^{\frac{\alpha+1}{2}}\\
\nonumber &\leq&
  C_{\t,\eps}\frac{(\alpha+1)^2}{2}\int_{\mathcal{C}(2R)}(|\nabla u|+|\nabla v|)^{p-2}[(u-v)^+]^{\alpha-1}|\nabla(u-v)^+|^2\varphi^{\frac{\alpha +1}{\alpha}}|\nabla \varphi|^{\frac{\alpha +1}{2\alpha}}\\
\nonumber &+&2 \int_{\mathcal{C}(2R)}|\nabla \varphi |^{\frac{\alpha+1}{2}}\,,
\end{eqnarray}
where $C_{\t,\eps}$ has been defined in \eqref{fkjngòsfnbncbvcnbfv}.

\

\noindent{\bf $\hookrightarrow$ Evaluation of the term $I_2$.}
We use the same notations as in the evaluation of $I_1$ and we get:
\begin{eqnarray}\label{Eq:I2}\\ \nonumber
&&\quad\quad I_2=\int_{\mathcal{C}(2R)} \left( [(u-v)^+]^{\frac{\alpha+1}{2}}\right)^2\varphi^2=
 \int_{\mathbb{R}^{N-1}}\left ( \int_{I_{x'}^{\tau,\eps}}\left([(u-v)^+]^{\frac{\alpha+1}{2}}\right)^2 dy \right)(\varphi(x'))^2 dx'\\
\nonumber &\leq& C_{\t,\eps} \left (\frac{\alpha+1}{2} \right )^2\int_{\mathbb{R}^{N-1}}
  \left( \int_{I_{x'}^{\t,\eps}} (|\nabla u|+|\nabla v|)^{p-2}[(u-v)^+]^{\alpha-1}|\partial_y(u-v)^+|^2 dy \right)(\varphi(x'))^2dx'\\
\nonumber &\leq& C_{\t,\eps} \left(\frac{\alpha+1}{2} \right )^2\int_{\mathcal{C}(2R)} (|\nabla u|+|\nabla v|)^{p-2}[(u-v)^+]^{\alpha-1}|\nabla(u-v)^+|^2\varphi^2\,.
\end{eqnarray}

Let us fix
\begin{equation}\label{eq:alpaaaaaaaaaaaa}
\alpha = 2N+1 >1.
\end{equation}

Recalling that
$C_{\t,\eps}$ tends to $0$, as both $\t$ and $\eps$ go to zero,
we can take  $\t >0$ and $\eps>0$ small enough, such that
\begin{equation}\label{eq:lambda}
c_2 C_{\tau,\eps}\left(\frac{\alpha+1}{2}\right)^2 < \frac 12, \qquad c_1C_{\t,\eps}(\alpha+1)^2 < 2^{-N}
\end{equation}
so that from \eqref{eq:jdladlsaaljkjlkjkljkldscquhdoufhwou} we have
\begin{equation}\label{eq:diff2bis}
\int_{\mathcal{C}(2R)} (|\nabla u|+|\nabla v|)^{p-2}|\nabla(u-v)^+|^2[(u-v)^+]^{\alpha-1}\varphi^2
\leq  2 c_1  I_1.
\end{equation}

From \eqref{Eq:Cut-off1}  we infer  that
\begin{eqnarray}\label{eq:diff3}
&&\int_{\mathcal{C}(R)}(|\nabla u|+|\nabla v|)^{p-2} |\nabla(u-v)^+|^2(u-v)^{\alpha-1}\\
\nonumber &\leq& \int_{\mathcal{C}(2R)} (|\nabla u|+|\nabla v|)^{p-2}|\nabla(u-v)^+|^2(u-v)^{\alpha-1}\varphi^2
\leq  2 c_1I_1
\end{eqnarray}
and, using \eqref{Eq:I1}, we obtain
\begin{eqnarray}\label{eq:cons}
&&\int_{\mathcal{C}(R)}(|\nabla u|+|\nabla v|)^{p-2}|\nabla(u-v)^+|^2[(u-v)^+]^{\alpha-1}\leq\\
\nonumber &\leq&c_1C_{\t,\eps}(\alpha+1)^2
  \int_{\mathcal{C}(2R)}
  (|\nabla u|+|\nabla v|)^{p-2}
  [(u-v)^+]^{\alpha-1}|\nabla(u-v)^+|^2\varphi^{\frac{\alpha +1}{\alpha}}|\nabla \varphi|^{\frac{\alpha +1}{2\alpha}}+\\
\nonumber &+& 4c_1\int_{\mathcal{C}(2R)}|\nabla \varphi|^{\frac{\alpha+1}{2}}.
\end{eqnarray}

Recalling \eqref{eq:alpaaaaaaaaaaaa} one has:
\begin{eqnarray}\label{eq:diff4}
&&\int_{\mathcal{C}(R)}(|\nabla u|+|\nabla v|)^{p-2}|\nabla(u-v)^+|^2(u-v)^{\alpha-1}\\
\nonumber &\leq&
  \theta \int_{\mathcal{C}(2R)}
  (|\nabla u|+|\nabla v|)^{p-2}
  |\nabla(u-v)^+|^2[(u-v)^+]^{\alpha-1}+\hat C  R^{-2},
\end{eqnarray}
where
$$
\theta=c_1C_{\t,\eps}(\alpha+1)^2,
$$
$$
\hat C= 4c_1\lambda C^{\frac{\alpha+1}{2}} >0
$$
exploiting also \eqref{Eq:Cut-off1}. Notice that, in view of \eqref{eq:lambda}, we also have that $\theta<2^{-N}$. In order to apply  Lemma \ref{Le:L(R)} we set
$$
\mathcal{L}(R)=\int_{\mathcal{C}(R)}(|\nabla u|+|\nabla v|)^{p-2}|\nabla(u-v)^+|^2[(u-v)^+]^{\alpha-1},
$$
and
$$ g(R)=\hat C R^{-2}.$$

Then from \eqref{eq:diff4} we have:
$$
\begin{cases}
\mathcal{L}(R)\leq \theta \mathcal{L}(2R)+g(R) & \forall R>0,\\
\mathcal{L}(R)\leq CR^{N} & \forall R >0.
\end{cases}
$$

Applying Lemma \ref{Le:L(R)} with $\beta=N$,
we get $\mathcal{L}(R)=0$ and consequently the thesis.
\begin{flushright}
$\square$
\end{flushright}

\noindent \textbf{Proof of Theorem \ref{th:COROwcpstrip}.}\\
The desired result is obtained with the same proof of that of Theorem \ref{th:wcpstrip} with the following slight (but necessary) modifications. Replace $\mathcal{S}_{(\t,\varepsilon)}$ by $\mathcal{S}$, set $ \varepsilon =\tau = \lambda$,  $ B_{x'}^\eps = \emptyset $,  $ I_{x'}^{\t,\eps}=A_{x'}^{\t} = \mathcal{S} \cap \{ x' \}\times (y_0-\frac{\lambda}{2}, y_0+\frac{\lambda}{2}) $ and observe that \eqref{fkjngòsfnbncbvcnbfv} becomes

\begin{equation}\label{LAMBDAfkjngòsfnbncbvcnbfv}
C_{\lambda}(x')\leq C_{\lambda}\,:=\, 2\lambda^2 (2 M_0)^{2-p},
\end{equation}
and that \eqref{eq:lambda} becomes

\begin{equation}\label{eq:lambdalambda}
c_2 C_{\lambda}\left(\frac{\alpha+1}{2}\right)^2 < \frac 12, \qquad c_1C_{\lambda}(\alpha+1)^2 < 2^{-N}.
\end{equation}
The conclusion the  follows by taking $\lambda$ small enough in the latter one.

\begin{flushright}
$\square$
\end{flushright}

\begin{remark} \label{valoriexpl}In view of \eqref{costante1},\eqref{costante2},\eqref{fkjngòsfnbncbvcnbfv}, \eqref{eq:alpaaaaaaaaaaaa} and of \eqref{eq:lambda}, it is possible to calculate explicitely the value of $\eps_0$ and $\tau_0$. The same can be done for $\overline \lambda$.
\end{remark}

\section{Properties of  the local symmetry regions}\label{sectioncompact}

Let us start this section recalling a useful change of variable. Under our assumptions on the function $ a=a(s)$ set
\[
A(s)\,:=\,\int_0^s\,a(t)^{\frac{1}{p-1}}\,dt\,.
\]

It follows that $A$ belongs to  $C^{1,1}_{loc} ([0,+\infty), [0,+\infty))$, is strictly increasing and satisfies $ A(0) = 0 $ and $ \lim_{t\to+\infty} A(t) = +\infty$. Set
\begin{equation}\label{kghksdhkhkghksresrrsr}
w\,:=\,A(u)\,
\end{equation}
then, $ w \in C^{1,\alpha}_{loc}(\overline{\mathbb{R}^N_+})$ and

\begin{equation}
\begin{cases}
-\Delta_p w\,+\,\tilde b(w)|\nabla w|^q\,=\,\tilde f(w), & \text{ in } \mathbb{R}^N_+\\
w(x',y) > 0, & \text{ in } \mathbb{R}^N_+\\
w(x',0)=0,&  \text{ on }\partial\mathbb{R}^N_+
\end{cases}
\end{equation}


\noindent with $\tilde b(w)\,:=\, \frac{b(A^{-1}(w))}{\big(a(A^{-1}(w))\big)^{\frac{q}{p-1}}}$ and $\tilde f(w)\,:=\, f(A^{-1}(w))$. It follows that  $\tilde b$ and $\tilde f$ still are locally Lipschitz continuous (here we used the fact that $ q> 1 > p-1$, since $(H_1)$ is in force). We notice that : $f(s)>0$ for $s>0$ if and only if $\tilde{f} (s)>0$ for $s>0$ and that, $ f(0) = \tilde{f}(0)$. This shows that, as we will recall later, it is not restrictive  to our purposes to assume from now on that
\begin{equation}\nonumber
a(s)=1\,.
\end{equation}

We also observe that using the  mean value theorem in the $y$-direction,
the Dirichlet condition on $u$ and the fact that $\n u \in L^{\infty}(\mathbb{R}^N_+)$,
we get that $u$ is bounded on every set of the form  $ \{0 \le y \le \eta\}$, for any $ \eta>0$.  This implies that also $w$ and $\nabla w$ are bounded on every set of the form  $ \{0 \le y \le \eta\}$, for any $ \eta>0$ (the bound might depend on $\eta$).

\medskip

With this in mind, now  we are  going to prove an important properties concerning the local symmetry regions of the solutions. Let us start with some notations.

We define the strip $\Sigma_\lambda$ by
\[
\Sigma_\lambda\,:=\, \{0<y<\lambda\}
\]
and the reflected function $u_\lambda(x)$ by
\[
u_\lambda(x)=u_\lambda(x',y)\,:=\, u(x',2\lambda-y)\,\quad\text{in}\,\,\, \Sigma_{2\lambda}\,.
\]



\noindent As customary we also define  the critical set $Z_u$  by
\[
Z_u=\{x\in \mathbb{R}^N_+ \,:\, \nabla u(x)=0\}.
\]

 We have the following
\begin{theorem}\label{trittofritto}
Let  $1<p<2$ and let $u \in C^{1,\alpha}_{loc}(\overline{\mathbb{R}^N_+})$ be a solution  to \eqref{E:P} with $a(s)=1$. Let $(H_1)$, $(H_2)$ be satisfied and assume that $f(s)>0$ for $s>0$.  Let us assume that both $u$ and $\nabla u$ are bounded on every strip $\Sigma_\eta$, $ \eta>0$. Suppose  furthermore  that $u$ is monotone non-decreasing in $\Sigma_\lambda$, for some $ \lambda>0$.
Assume that $\mathcal{U}$ is a (not empty) connected component of $\Sigma_\lambda\setminus Z_u$  such that $u(x)\equiv u_{\lambda}(x)$ in $\mathcal{U}$, (i.e. a local symmetry region for $u$).
Then   necessarily $\mathcal{U}$ touches the boundary of ${\mathbb{R}^N_+}$, namely
$$\partial \mathcal{U}\cap\{y=0\}\,\neq\,\emptyset.$$
\end{theorem}
\begin{proof}
\noindent Let us start showing the following
\begin{center}
{\em Claim 1:} There exists $\t=\t(\mathcal{U},\lambda)>0$ such that $\text{dist}(\mathcal{U},\{y=0\})>\t$.
\end{center}
By contradiction let us assume that there exists a sequence of points
\[
x_n=(x'_n,y_n)\in \mathcal U,
\]
such that
\begin{equation}\label{dgfygikjvb}
\underset{n\rightarrow \infty}{\lim} \text{dist}(x_n\,,\,\{y=0\})\,=\underset{n\rightarrow \infty}{\lim} y_n=0.
\end{equation}
We consider the  sequence
$$
\hat x_n=(x'_n,\frac{\lambda}{2})
$$  and  the two different cases
\begin{itemize}
\item[$a)$]  $u(\hat x_n)$ is strictly bounded away from zero uniformly on $n$;
\item [$b)$] up to subsequences $\underset{n\rightarrow \infty}{\lim} u(\hat x_n)=0$.
\end{itemize}

\

\textbf{Case $a)$.}  Define the sequence
\begin{equation}\label{hdjshdjshjjsdfghdeyrty}
u_n(x)= u(x'+x'_n,y).
\end{equation}
Let $\mathcal K \subset{\overline {\mathbb{R}_N^+}}$ be a compact set. Since both $u$ and $\nabla u$ are bounded on every strip $\Sigma_\eta$, $ \eta>0$, we get :
$$
0 \le u_n(x)\leq C({\mathcal K}), \qquad \vert \nabla u_n(x) \vert \leq C({\mathcal K}), \qquad \forall x \in \mathcal K,
$$
for some positive constant $C({\mathcal K})$.

Therefore  $C^{1,\alpha}$ estimates (see the classical results \cite{Di,Li,T}),
Ascoli's Theorem and a standard diagonal process imply that
\begin{equation}\label{hdjshdjshjjsdfghdeyrty11}
{u}_n\overset{C^{1,\alpha'}_{loc}(\overline{\mathbb{R}^N_+})}{\longrightarrow}u_\infty.
\end{equation}
up to subsequences,  for  $\alpha '<\alpha$.
Recalling that $u_n(0,\frac{\lambda}{2})\geq \gamma_0>0$,
uniform convergence implies that $u_\infty $ is a non-trivial non-negative solution to the equation in \eqref{E:P} (with $a(s)=1$).
Actually, by the strong maximum principle \cite{V}, we have that
\begin{equation}\label{dhjkshjjIXXEUNYBRCRTUW}
u_{\infty}(x)>0\quad \forall x \in \mathbb{R}^N_+.
\end{equation}
By the definition of $\mathcal U$  (i.e. $u(x)\equiv u_{\lambda}(x)$ in $\mathcal U$), since  by \eqref{dgfygikjvb} (together with the  Dirichlet condition and  the  mean value theorem) we have
$$
u(x_n',y_n)\leq \|\nabla u(x)\|_{\infty}\cdot y_n \rightarrow 0, \quad \text {as }n\rightarrow +\infty,
$$
then it follows:
\begin{equation}\label{eqkjsakdsksahgjghjgh}\lim_{n\rightarrow +\infty}u(x_n', 2\lambda-y_n)=\lim_{n\rightarrow +\infty} u(x_n',y_n)=0.\end{equation}
Then from  \eqref{eqkjsakdsksahgjghjgh},  using \eqref{hdjshdjshjjsdfghdeyrty} and \eqref{hdjshdjshjjsdfghdeyrty11},
we obtain:
$$
u_{\infty}(0,2\lambda)=0.
$$
This is a contradiction with \eqref{dhjkshjjIXXEUNYBRCRTUW} and \emph{Claim 1} is proved in case $a)$.
\

\textbf{Case $b)$.} Arguing exactly as in the proof of the case $a)$ here above, it follows that necessarily $u_\infty\equiv 0$. This implies that case $b)$  occurs only if $f(0)=0$ because if not, $0$ can not be a solution of our equation. \\
\noindent Recalling \eqref{hdjshdjshjjsdfghdeyrty} we define
\begin{equation}\label{skajskajkjsommomoininubu}
\bar{u}_n(x',y)\equiv\frac{u_n(x',y)}{u_n(0,\frac{\lambda}{2})}=\frac{u(x'+x'_n,y)}{u(x'_n,\frac{\lambda}{2})},
\end{equation}
so that
$$
\bar{u}_n(0,\frac{\lambda}{2})=1,
$$
and $u_n$ uniformly converges to $0$ on compact sets of $\overline{\mathbb{R}^N_+}$ by construction. Recalling that we are assuming that $a(s)=1$,
it is easily seen that
$$
-{\hbox {\rm div}} (|\nabla \bar{u}_n|^{p-2} \nabla \bar{u}_n )+\big(u_n(0,\frac{\lambda}{2})\big)^{q-(p-1)}\cdot b(u_n)\cdot|\nabla \bar{u}_n|^q=\frac{f(u_n)}{u_n^{p-1}}\cdot \bar{u}_n^{p-1},
$$
that is
\begin{equation}\label{equadifff}
-{\hbox {\rm div}} (|\nabla \bar{u}_n|^{p-2} \nabla \bar{u}_n )+\hat c_n(x)\cdot|\nabla \bar{u}_n|^{p-1}=\tilde c_n(x)\cdot \bar{u}_n^{p-1},
\end{equation}
where
$$
\hat c_n (x) = b(u_n) \vert \nabla u_n \vert^{q - (p-1)},  \qquad
\tilde c_n(x)=\frac{f(u_n)}{u_n^{p-1}}.
$$
The assumptions on $b(\cdot)$ and $f(\cdot)$ and the fact that $q>1> p-1$
imply that $\hat c_n$ and $\tilde c_n$ are bounded on every strip  $\Sigma_\eta$, $ \eta>0$. Indeed, since $b$ and $f$ are locally Lipschitz continuous functions we have

\begin{equation} \label{eq:222a}
\vert \hat c_n(x)\vert \le \Vert b(u) \Vert_{\infty, \Sigma_\eta} \Vert \nabla u \Vert^{q-(p-1)}_{\infty, \Sigma_\eta}
\end{equation}

\begin{equation} \label{eq:22222222222}
0 \le \tilde c_n(x)=\frac{f(u_n)}{u_n^{p-1}} \le L_f(\Vert u\Vert_{\infty, \Sigma_\eta} ) \vert u_n(x)  \vert^{2-p} \le L_f(\Vert u\Vert_{\infty, \Sigma_\eta} )\Vert u \Vert^{2-p}_{\infty, \Sigma_\eta},
\end{equation}
where we used  that $f(0) =0$.

On the other hand we have that $\underset{t\rightarrow 0^+}{\lim}\frac{f(t)}{t^{p-1}}=0$, since $f(0)=0$ as remarked above, $p-1< 1$ and  $f$ is locally Lipschitz continuous. The latter, together with the fact that both $u_n$ and $ \nabla u_n$ converge to zero uniformly on compact sets of $\overline{\mathbb{R}^N_+}$ and $q>1> p-1$, immediately yields that also $\tilde c_n$ and $ \hat c_n$ converge to zero uniformly on compact sets of $\overline{\mathbb{R}^N_+}$.

\medskip

Fix $ 0< \delta < \frac{\lambda}{2}$ and consider un arbitrary compact set ${\mathcal{K}}$ of $\overline{\mathbb{R}^N_+}$ containing the point $(0, \frac{\lambda}{2})$.  By Harnack inequality (see \cite[Theorem 7.2.2]{PSB}), applied to the equation \eqref{equadifff} on the compact set ${\mathcal{K}\cap \{y\geqslant \delta\}} \subset \subset  {\mathbb{R}^N_+}$, we get that:
\begin{equation}\label{H1bis}
\underset{\mathcal{K}\cap \{y\geqslant \delta\}}{\sup} {\bar u}_n \leqslant C_H\underset{\mathcal{K}\cap \{y\geqslant \delta\}}{\inf} {\bar u}_n
\leqslant C_H.
\end{equation}
Moreover, by the monotonicity of $u$ in $\Sigma_{\lambda}$ we have:
\begin{equation}
\underset{\mathcal{K}\cap\{y\geqslant 0\}}{\sup}\,\, {\bar u}_n \leqslant \underset{\mathcal{K}\cap \{y\geqslant \delta\}}{\sup} {\bar u}_n \leqslant C_H.
\end{equation}
Hence we can use (once again) $C^{1,\alpha}$ estimates,  Ascoli's Theorem and a diagonal argument to get, up to subsequences, that:
\[
\bar{u}_n\overset{C^{1,\alpha'}_{loc}({\overline {\mathbb{R}_+^N}})}{\longrightarrow}\bar{u}
\]
for  $\alpha '<\alpha$.
Arguing as above,  we see that $\bar{u}\geqslant 0 $ in ${\overline {\mathbb{R}^N_+}}$
and $\bar{u}(0,\frac{\lambda}{2})= 1$.\\
\noindent Taking into account  the properties of $\hat c_n$ and of $\tilde c_n$,
we can pass to the limit in \eqref{equadifff} obtaining:
$$
- \Delta_p \bar{u}=0\qquad \text{in}\quad \mathbb{R}^N_+.
$$
By the Strong Maximum Principle \cite{V}, we therefore get that $\bar{u}>0$ since
$\bar{u}$ cannot be equal to zero because of the condition :
$\bar{u}(0,\frac{\lambda}{2})=1$.
Actually, by construction, we have
\begin{equation}\label{E:Ptrici}
\begin{cases}
-\Delta_p \bar{u}=0, & \text{ in }\mathbb{R}^N_+\\
\bar{u} > 0, & \text{ in } \mathbb{R}^N_+\\
\bar{u}(x',0)=0,&  \text{ on }\partial\mathbb{R}^N_+.
\end{cases}
\end{equation}
By results in \cite{KSZ} it follows that $\bar{u}$ is affine linear, that means:
\[
\bar{u}(x',y)=ky,
\]
for some $k>0$ by the Dirichlet assumption.
This is a contradiction since by assumption $u$ and consequently $\bar{u}$ have a local symmetry region
and this concludes the proof of the {\em Claim 1}.

\

\noindent We show the following
\begin{center}
{\em Claim 2:} There exists $\gamma=\gamma(\mathcal{U},\lambda)>0$ such that $u\geq \gamma$ in $\mathcal U$.\end{center}
To show this, assume by contradiction that there exists a sequence of points
\[
x_n=(x'_n\,,\,y_n)\in\mathcal U,
\]
such that
\begin{equation}\label{dgfygikjvbbhfd}
\underset{n\rightarrow \infty}{\lim} u(x_n)=0
\end{equation}
and with  $y_n$ converging (up to subsequences) to $y_0>0$.

Using definition \eqref{hdjshdjshjjsdfghdeyrty}  we set
\[
\bar{u}_n(x',y)\equiv\frac{u(x'+x'_n,y)}{u(x'_n,y_n)},
\]
so that $\bar{u}(0,y_n)=1$ and $u_n$  uniformly converges to $0$ on compact sets of
$ {\overline {\mathbb{R}_+^N}}$ by construction (see \eqref{dgfygikjvbbhfd}).
As above (see \eqref{equadifff}), it is easy to see that
\begin{equation}\label{equadifffvddfsfsf}
-{\hbox {\rm div}} (|\nabla \bar{u}_n|^{p-2} \nabla \bar{u}_n )+\hat c_n(x)\cdot|\nabla \bar{u}_n|^{p-1}=\tilde c_n(x)\cdot \bar{u}_n^{p-1},
\end{equation}
with $\hat c_n$ and $\tilde c_n$ satisfy \eqref{eq:222a} and \eqref{eq:22222222222} and, both of them converges to zero uniformly on compact sets of $ {\overline {\mathbb{R}_+^N}}$.  Furthermore, we have that

\[
\bar{u}_n\overset{C^{1,\alpha'}_{loc}({\overline {\mathbb{R}_+^N}})}{\longrightarrow}\bar{u}
\]
up to subsequences,  for  $\alpha '<\alpha$.
Then, arguing as above,  we get that $\bar{u}\geqslant 0 $ in ${\overline {\mathbb{R}^N_+}}$
and $\bar {u}(0,y_0)= 1$, with  $y_0>0$. Moreover
$$- \Delta_p \bar{u}=0\qquad \text{in}\quad \mathbb{R}^N_+.$$
By the Strong Maximum Principle (see \cite{V}) we  get that
$\bar{u}>0$, since $\bar{u}$  cannot be equal to $0$ because of the condition:
$\bar{u}(0,y_0)=1$.
In fact, as in \eqref{E:Ptrici}, by construction and
using results in \cite{KSZ}, it follows  that $\bar{u}$ is of the form :
\begin{equation}\label{fdgsdgdkhgdhfgdhg}
\bar{u}(x',y)=ky,
\end{equation}
for some $k>0$.

Since by construction
\[
\bar u (0,y_0)=\bar u_\lambda (0,y_0),
\]
if $y_0<\lambda$, we get a contradiction by \eqref{fdgsdgdkhgdhfgdhg}
 concluding the proof of {\em Claim 2}. If else $y_0=\lambda$, it follows by construction that $\frac{\partial\,\bar u}{\partial y} (0,y_0)=0$. This is deduced by observing that $\frac{\partial\,\bar{u}_n}{\partial y}$ vanishes somewhere on the segment from $(0,y_n)$ to $(0,2\lambda-y_n)$ (since $(x'_n\,,\,y_n)\in\mathcal U$) and exploiting the uniform convergence of the gradients. Therefore again we get a contradiction by \eqref{fdgsdgdkhgdhfgdhg}
 concluding the proof of {\em Claim 2}.

\

\noindent Since $f(s)>0$ for $s>0$, {\em  Claim 2} implies that there exists $\gamma^+ >0$ such that
\begin{equation}\label{fffffff}
f(u)\geq\gamma^+ \quad \text{in } \mathcal U.
\end{equation}
Now we proceed in order to conclude the proof.
Let $\varphi_R(x',y)=\varphi_R(x')$ with $\varphi(x') \in C^{\infty}_c (\mathbb{R}^{N-1}) $  defined as in \eqref{Eq:Cut-off1}.
For all $\varepsilon>0$, let $G_\varepsilon:\mathbb{R}^+\cup\{0\}\rightarrow\mathbb{R}$ be defined as:
\begin{equation}\label{eq:G}
G_\varepsilon(t)=\begin{cases} t, & \text{if  $t\geq 2\varepsilon$} \\
2t-2\varepsilon,& \text{if $\varepsilon\leq t\leq2\varepsilon$}
\\ 0, & \text{if $0\leq t\leq \e$}.
\end{cases}
\end{equation}
\noindent Let $\chi$ be the characteristic function of a set .
We define
$$
\Psi=\Psi_{\varepsilon,R}\,:=\,e^{-s(u)}\displaystyle \varphi_R\frac{G_\varepsilon(|\nabla u|)}{|\nabla u|}\chi_{(\mathcal U\cup \mathcal{U}^\lambda)},
$$
where $\mathcal{U}^\lambda$ is the reflected set of $\mathcal{U}$ w.r.t. the hyperplane $T_\lambda=\{y=\lambda\}$ and
\begin{equation}\label{eq:aooooooooo}
s(t)=\hat C\cdot \int_0^t \, b^+(t')dt',
\end{equation}
with $\hat C$ some positive constant to be chosen later.
By the definition of  $\mathcal U$ and taking into account the fact that $\mathcal{U}$ is a local symmetry region of $u$,     we have that $\nabla u =0$ on $\partial(\mathcal U\cup \mathcal{U}^\lambda)$.
Moreover $\n u \in L^{\infty}(\mathbb{R}^N_+)$ implies that $u$ is bounded in $\Sigma_{\lambda}$.
Therefore  $\Psi$ is well defined and we can use it as test function in equation \eqref{E:P} (see also \cite[Lemma 5]{MMPS}),
getting:
\begin{eqnarray}\label{qqqqqqbissete}
&&\int_{\mathcal U\cup \mathcal{U}^\lambda}\, |\nabla u|^{p-2} (\nabla u ,\nabla\Psi) dx+
\int_{\mathcal U\cup \mathcal{U}^\lambda}b^+(u)|\nabla u|^q\Psi dx=\int_{\mathcal U\cup \mathcal{U}^\lambda}b^-(u)|\nabla u|^q\Psi dx\\\nonumber&+&\int_{\mathcal U\cup \mathcal{U}^\lambda}f(u)\Psi dx.
\end{eqnarray}
Since $u$ and $\Psi$ are even w.r.t. the hyperplane $\{y=\lambda\}$, it follows that $(\nabla u ,\nabla\Psi)$ is even too. Therefore we infer  that
\begin{equation}\label{qqqqqq}
\int_{\mathcal U}\, |\nabla u|^{p-2} (\nabla u ,\nabla\Psi) dx+
\int_{\mathcal U}b^+(u)|\nabla u|^q\Psi dx=\int_{\mathcal U}b^-(u)|\nabla u|^q\Psi dx+
\int_{\mathcal U}f(u)\Psi dx.
\end{equation}
Let us suppose $1\leq q\leq p$. For every $\sigma >0$  we have:
\begin{equation}\label{blelellelelle}
x^q\leq C(\sigma) \cdot x^p+\sigma,\quad x\geq 0\,,
\end{equation}
say e.g. $ C(\sigma)=\sigma^{1-\frac{p}{q}}$.
Therefore
\eqref{qqqqqq} and \eqref{blelellelelle} imply:
\begin{eqnarray}\label{qeqeqeeqeqeq}
&&\int_{\mathcal U}\, |\nabla u|^{p-2} (\nabla u ,\nabla\Psi) dx+C(\sigma)\int_{\mathcal U}b^+(u)|\nabla u|^p\Psi dx+\sigma\int_{\mathcal U}b^+(u)\Psi dx\\\nonumber
&&\geq\int_{\mathcal U}b^-(u)|\nabla u|^q\Psi dx+\int_{\mathcal U}f(u)\Psi dx\geq \int_{\mathcal U}f(u)\Psi dx.
\end{eqnarray}
By \eqref{fffffff} we can choose $\sigma$ in \eqref{blelellelelle}, say $\bar \sigma$,
small enough such that
$$
\gamma^+-\bar\sigma \,\|b^+(u)\|_{\infty}=\tilde C>0\,,
$$
so that
\begin{eqnarray}\label{qeqeqeeqeqeqAAA}
&&\int_{\mathcal U}\, |\nabla u|^{p-2} (\nabla u ,\nabla\Psi) dx+C(\bar \sigma)\int_{\mathcal U}b^+(u)|\nabla u|^p\Psi dx\geq\tilde C\int_{\mathcal U}\Psi dx.
\end{eqnarray}
Chosing $\hat C$ in \eqref{eq:aooooooooo} equal to $C(\bar \sigma)$ in \eqref{qeqeqeeqeqeqAAA} we obtain
\begin{eqnarray}\label{VaScOoOoOoO}\\\nonumber
&&\int_{\mathcal U}e^{-s(u)}\varphi_{R}|\nabla u|^{p-2}(\nabla u,  \nabla \frac{G_\varepsilon(|\nabla u|)}{|\nabla u|})dx\\\nonumber
&&+\int_{\mathcal U}e^{-s(u)}\frac{G_\varepsilon(|\nabla u|)}{|\nabla u|}
|\nabla u|^{p-2}(\nabla u, \nabla \varphi_{R})dx\\\nonumber
&&\geq\tilde C\int_{\mathcal U} e^{-s(u)}\displaystyle \varphi_R\frac{G_\varepsilon(|\nabla u|)}{|\nabla u|} dx
\end{eqnarray}
We set ${\displaystyle h_\varepsilon (t)=\frac{G_\varepsilon(t)}{t}}$, meaning that $h(t)=0$ for $0\leq t\leq \varepsilon$.
We have:
\begin{eqnarray}\label{eq:smm3}
&&\left| \int_{\mathcal U}\,e^{-s(u)}\varphi_{R}|\nabla u|^{p-2}(\nabla u,  \nabla \frac{G_\varepsilon(|\nabla u|)}{|\nabla u|})dx\right|\\\nonumber &&\leq C \int_{\mathcal U}|\nabla u|^{p-1}|h_\varepsilon'(|\nabla u|)||\nabla (|\nabla u|)|\varphi_R dx\\\nonumber
&&\leq C\int_{\mathcal U}|\nabla u|^{p-2}\Big(|\nabla u|h_\varepsilon'(|\nabla u|)\Big)\|D^2 u\| \varphi_R dx,
\end{eqnarray}
where $\|D^2 u\|$ denotes the Hessian  norm.\\

\noindent Here below, we fix $R>0$ and let $\varepsilon\rightarrow 0$. Later we will let $R\rightarrow \infty$.
To this aim, let us first show  that
\begin{itemize}
\item [$(i)$] $|\nabla u|^{p-2}||D^2 u|| \varphi_R \in L^1(\mathcal U) \,\,\, \forall R>0$;

\

\item [$(ii)$] $|\nabla u|h_\varepsilon'(|\nabla u|)\rightarrow 0$ a.e. in $\mathcal U$ as $\varepsilon \rightarrow 0$ and $|\nabla u|h_\varepsilon'(|\nabla u|)\leq C$ with $C$ not depending on $\varepsilon$.
\end{itemize}
Let us  prove $(i)$. Defining
$\mathcal{D}(R)=\left\{ \mathcal{U} \cap {\{B^{'}(0,R)\times \mathbb{R}\}} \right\},$ by H\"older's inequality it follows
\begin{eqnarray}\label{eq:smm4}
&&\int_{\mathcal U}|\nabla u|^{p-2}||D^2u||\varphi_R dx\leq C(\mathcal{D}(2R))\left( \int_{\mathcal{D}(2R)}|\nabla u|^{2(p-2)}||D^2u||^2\varphi_R^2 dx \right)^{\frac 12}\\\nonumber
&\leq&C\left( \int_{\mathcal{D}(2R)}|\nabla u|^{p-2-\beta}||D^2u||^2\varphi_R^2|\nabla u|^{p-2+\beta}dx \right )^{\frac 12}\\\nonumber &\leq& C ||\nabla u||^{(p-2+\beta)/2}_{L^{\infty}(\mathbb{R}^N_+)} \left(\int_{\mathcal{D}(2R)}|\nabla u|^{p-2-\beta}||D^2 u||^2 dx \right)^{\frac 12},
\end{eqnarray}
with  $0\leq\beta<1$ and $\varphi^2_R|\nabla u|^{p-2+\beta}$ consequently bounded .
By {\em Claim 1} we have:
$$
\text{dist}(\mathcal{U},\{y=0\})>0.
$$
Using  \cite[Proposition 2.1]{MSS}\footnote{Actually in Proposition 2.1 of \cite {MSS} it is considered the case $q=p$. The same result in the more general case $1<q\leq p$ follows exactly in the same way repeating the same calculations.}, we infer that
$$\left(\int_{\mathcal{D}(2R)}|\nabla u|^{p-2-\beta}||D^2 u||^2 dx \right)^{\frac 12}\leq C.$$ Then by \eqref{eq:smm4} we obtain
$$\int_{\mathcal U}|\nabla u|^{p-2}||D^2u||\varphi_R dx\leq C.$$

Let us  prove $(ii)$. Recalling \eqref{eq:G}, we obtain
$$
h'_\varepsilon(t)=
\begin{cases} 0 & \text{if  $t\geq 2\varepsilon$} \\
\frac{2\varepsilon}{t^2}& \text{if $\varepsilon\leq t\leq2\varepsilon$}
\\ 0 & \text{if $0\leq t\leq \e$},
\end{cases}
$$
and then $|\nabla u|h_\varepsilon'(|\nabla u|)$ tends to $0$
almost everywhere in $\mathcal U$ as $\varepsilon$ goes to $0$ and we have:
$|\nabla u|h_\varepsilon'(|\nabla u|)\leq 2$.
\

\noindent Then by \eqref{VaScOoOoOoO}, \eqref{eq:smm3} and $(i),$ $(ii)$ above,  passing to the limit
as $\varepsilon \rightarrow 0$, we get:
$$
\int_{\mathcal U}\,e^{-s(u)}
|\nabla u|^{p-2}(\nabla u, \nabla \varphi_{R})dx
\geq C\int_{\mathcal U}\varphi_Rdx, \quad \forall R>0.
$$
Recalling \eqref{Eq:Cut-off1},
we have that there exists $C=C(\|\nabla u\|_{L^{\infty}(\mathbb{R}^N_+)})$ (not depending on $R$)
such that:
$$
|\mathcal U \cap \text{supp}(\varphi_{R})|\cdot \frac 1R\geq C\cdot |\mathcal U \cap \text{supp}(\varphi_{R})|
$$
and we get a contradiction for $R$ large, concluding the proof.
\end{proof}
\section{Recovering Compacteness}\label{sectioncompactbistreremdbvfj}
In this section we prove a crucial result, which allows us to localize the support of $(u-u_{\bar \lambda})^+$, where  $\bar\lambda$ is defined in \eqref{barlambda} below. The localization obtained will enable us to apply the weak comparison principle Theorem \ref{th:wcpstrip}.

With the notations introduced at the beginning of the previous section, we set
\begin{equation}\label{MP}
\Lambda \,:=\,\Big\{\lambda\in\mathbb{R}^+\,\,:\,\, u\leq u_\mu \,\,\text{in}\,\, \Sigma_\mu \,\,\,\forall \mu<\lambda\Big\}\,
\end{equation}
and we define
\begin{equation}\label{barlambda}
\bar \lambda \,:=\, \sup\,\Lambda\,.
\end{equation}


We have the following:
\begin{proposition}\label{lellalemma} Let  $1<p<2$ and let $u \in C^{1,\alpha}_{loc}(\overline{\mathbb{R}^N_+})$ be a solution  to \eqref{E:P} with $a(s)=1$. Let $(H_1)$, $(H_2)$ and $(H_3)$ be satisfied and assume that $f(s)>0$ for $s>0$.  Let us assume that both $u$ and $\nabla u$ are bounded on every strip $\Sigma_\eta$, $ \eta>0$.

Assume $ 0 < \bar\lambda < + \infty$  and set $$W_\varepsilon:=\Big(u-u_{\bar \lambda+\varepsilon}\Big)\cdot \chi_{\{y\leqslant \bar{\lambda}+\vep\}}$$
where $ \varepsilon >0$.

Given $0<\delta <\frac{\blambda}{2}$ and $\rho>0$, there exists $\vep_0 >0 $ such that, for any $\vep\leqslant \vep_0$, it follows
\[
\text{Supp} \,W_\vep^+\subset \{0\leqslant y\leqslant \delta\}\cup\{\bar{\lambda}-\delta\leqslant y\leqslant\bar{\lambda}+\vep\}\cup \left(\bigcup_{x'\in\mathbb R^{N-1}} B_{x'}^\rho\right).
\]
where $B_{x'}^\rho$ is such that
\begin{equation}\label{njgkdngkfkkjc}
B_{x'}^\rho\subseteq\left\{y\in(0,\bar\lambda+\varepsilon)\,:\, |\nabla u(x',y)|<\rho,\,|\nabla u_{\bar \lambda +\varepsilon}(x',y)|<\rho\right\}.
\end{equation}
\end{proposition}
\begin{proof}

Assume by contradiction that there exists $\delta >0$, with $0<\delta <\frac{\blambda}{2}$,
such that, given any $\vep_0>0$, we find $\vep\leqslant \vep_0$ and
$x_\vep=(x'_\vep,y_\vep)$ such that:
\begin{itemize}
\item [$(i)$]  $u(x'_\vep,y_\vep)\geqslant u_{\bar{\lambda}+\vep}(x'_\vep,y_\vep)$
\item  [$(ii)$] $x_\varepsilon$ belongs to the set
$$
\Big\{(x',y) \in\R^N : \delta\leqslant y_\vep\leqslant \bar{\lambda}-\delta\Big\} $$

and it holds the alternative: either $|\nabla u (x_\vep)|\geq\rho$ or $|\nabla u_{\bar \lambda +\varepsilon}(x_\varepsilon)|\geq\rho
$.
\end{itemize}

Take now $\vep_0=\frac{1}{n}$, then there exists $\vep_n \leqslant \frac{1}{n}$  and
a sequence $$
x_n=(x'_{n},y_{n})=(x'_{\vep_n},y_{\vep_n})
$$
such that
$$
u(x'_n,y_n)\geqslant u_{\bar{\lambda}+\vep_n}(x'_n,y_n)
$$
and satisfying conditions $(i)$ and $(ii)$ above. Up to subsequences we may assume that:
$$
y_n\rightarrow y_0 \quad \text{as} \quad n\rightarrow +\infty,\qquad
\text{ with } \quad\delta\leqslant y_0\leqslant \bar{\lambda}-\delta\,.
$$
Let us define
\begin{equation}\label{cazzo}
\tu_n(x',y)=u(x'+x'_n,y),
\end{equation}

Since both $u$ and $\nabla u$ are bounded on every strip $\Sigma_\eta$, $ \eta>0$, as before, by $C^{1,\alpha}$ estimates, Ascoli's Theorem and a standard diagonal process we get that :

\begin{equation}\label{roccaseccatrubis}
\tu_n\overset{C^{1,\alpha'}_{loc}({\overline {\mathbb{R}^N_+}})} {\longrightarrow}\tu
\end{equation}
(up to subsequences) for $\alpha '<\alpha$.
\

\noindent \emph{We claim that}
\begin{itemize}
\item[-] $\tu\geqslant 0 $ in $\mathbb{R}^N_+$, with $\tu(x,0)=0$ for every $x \in \mathbb{R}^{N-1}$;
\item [-] $\tu\leqslant \tu_{\blambda}$ in $\Sigma_{\blambda};$
\item[-] $\tu(0,y_0)= \tu_{\blambda}(0,y_0)$;
\item [-] $|\nabla \tu (0,y_0)|\geq \rho$.
\end{itemize}

\noindent To prove this note that, since each $\tu_n(x',y)$ is positive and satisfies the homogeneous Dirichlet boundary condition by construction, we have: $\tu\geqslant 0 $ in $\mathbb{R}^N_+$ and $\tu(x,0)=0$ for every $x \in \mathbb{R}^{N-1}$. It is also clear that $\tu\leqslant \tu_{\blambda}$ in $\Sigma_{\blambda}$ and $\tu(0,y_0)\geqslant \tu_{\blambda}(0,y_0)$.
Since (as shown above) $\tu\leqslant \tu_{\blambda}$,
actually there holds: $\tu(0,y_0)= \tu_{\blambda}(0,y_0)$.
Finally, at $x_0=(0,y_0)$
(where $\tu(0,y_0)= \tu_{\blambda}(0,y_0)$)
we have that $\nabla \tilde u(0,y_0)=\nabla\tilde u_{\bar \lambda}(0,y_0)$,
because $x_0$ is an interior minimum point for the function
$w(x):=\tilde u_{\blambda}(x)-\tilde u(x)\geq 0$.
For all $n$ we have
$|\nabla u (x_n)|\geq\rho$ or $|\nabla u_{\bar \lambda +\varepsilon_{n}}(x_n)|\geq\rho$,
and, using the uniform~$C^1$ convergence on compact set,  we get: $|\nabla \tu (0,y_0)|\geq \rho$.

\

Recalling that we assumed here $a(s)=1$, passing to the limit we obtain that $\tilde u$ satisfies
\begin{equation}\nonumber
\int_{\mathbb{R}^N_+}|\nabla \tilde{u}|^{p-2}(\nabla \tilde{u},\nabla \varphi)dx+\int_{\mathbb{R}^N_+}b(\tilde u)|\nabla \tilde u|^qdx=\int_{\mathbb{R}^N_+}f(\tilde{u})\varphi dx \quad \forall \varphi \in C^{\infty}_c(\mathbb{R}^N_+).
\end{equation}
\noindent Since  $\tu \geqslant 0$ in $\mathbb{R}^N_+$,
by the strong maximum principle \cite[Theorem 2.5.1]{PSB},
it follows that $\tu >0$ or $\tu=0$:
by the fact that $|\nabla \tu (0,y_0)|\geq \delta$, the case $\tu=0$ is not possible. Hence $\tu >0$ on $\mathbb{R}^N_+$.  Moreover we have  $\tu\leqslant \tu_{\blambda}$ in $\Sigma_{\blambda}$
and   $\tu(0,y_0)= \tu_{\blambda}(0,y_0)$. By the strong comparison principle \cite[Theorem 2.5.2]{PSB}, we have that
$\tu= \tu_{\blambda}$ in the connected component, say  $\mathcal{U}$, of $\Sigma_{\bar\lambda}\setminus Z_u$ containing the point $(0,y_0)$. By Theorem \ref{trittofritto} we have that $\partial \mathcal{U}\cap\partial \mathbb R^N_+ \neq \emptyset$.  The latter yields the existence of a point $ z=(z', 2 {\blambda} ) $ such that $ \tu(z) = 0$, which  contradicts $\tu >0$ in $\mathbb R^N_+$.
\end{proof}

\section{Proof of Theorem \ref {mainthm}}\label{sect.cort}

Let us start recalling that, in view of the changing of variable  in \eqref{kghksdhkhkghksresrrsr}, which preserves the monotonicity property, it is not restrictive to prove Theorem \ref {mainthm} in the case $a(\cdot)=1$. As already remarked, the assumption $\n u \in L^{\infty}(\mathbb{R}^N_+)$ implies that that $w$ and $\nabla w$ (see \eqref{kghksdhkhkghksresrrsr}) are bounded on every set of the form  $ \{0 \le y \le \lambda\}$, for any $ \lambda >0$. Thus, we can use the results demonstrated in Sections 3 and Sections 4.

The proof is based on the moving planes procedure. By Theorem \ref{th:COROwcpstrip} the set $\Lambda$ defined in \eqref{MP} is not empty and $ \bar\lambda \in (0, +\infty]$. To conclude the proof we need to show that $\bar\lambda=\infty$.\\

Assume that $ \bar\lambda$ is finite, set $ \lambda_0= {\bar\lambda}+2 $,
\[
M_0\,:=  \Vert u\Vert_{L^\infty(\{ 0 \le y \le {2\bar\lambda}+10 \})} + \Vert \nabla u \Vert_{L^\infty(\{ 0 \le y \le {2\bar\lambda}+10 \})} +1  >0
\]

\noindent and take $\tau_0=\tau_0(N,p,q,\lambda_0, M_0, \gamma)>0$
and $\eps_0=\eps_0(N,p,q,\lambda_0, M_0, \gamma)>0$ as in Theorem \ref{th:wcpstrip}.

By Proposition  \ref{lellalemma} we have that, given $0<\delta < \min \{\frac{\blambda}{2}, \frac{\tau_0}{4}\}$ and
$0< \rho < \eps_0$,  we find $\bar \e>0$ such that, for any
$0< \vep\leqslant \min \{\bar\e, \frac{\tau_0}{4},1\}$, it follows
\[
\text{Supp} \,W_\vep^+\subset \{0\leqslant y\leqslant \delta\}
\cup\{\bar{\lambda}-\delta\leqslant y\leqslant\bar{\lambda}+\vep\}\cup \left(\bigcup_{x'\in\mathbb R^{N-1}} B_{x'}^\rho\right),
\]
where $W_\vep^+ =(u-u_{\bar{\lambda}+\vep})^+\cdot\chi_{\{y\leqslant
\bar{\lambda}+\vep\}}$ and $B_{x'}^\rho$ is defined in \eqref{njgkdngkfkkjc}.

{\em We claim that $ u \le u_{\bar{\lambda}+\vep} $ in $\Sigma_{\bar\lambda+ \vep}$}, which contradicts the definition of $\blambda$ and yields that $\blambda=\infty$. This, in turn,  implies the desired monotonicity of $u$, that is $ \frac{\partial u}{\partial y}(x',y)\geqslant 0$ in  $\mathbb{R}^N_+ $.

To this end, we proceed by contradiction. Suppose that the open set

\[
\mathcal{S}_{(2\delta + \vep,\rho)} := \{ x \in \Sigma_{\bar\lambda+ \vep} \, : \, u(x) - u_{\bar{\lambda}+\vep}(x) >0\}
\]
is not empty, then $u$ and $ v =  u_{\bar{\lambda}+\vep} $ satisfy \eqref{Eq:WCP} with
$\lambda = y_0 = {\bar{\lambda}} +\vep \, ( < \lambda_0)$, as well as: $ \Vert u \Vert_{\infty} + \Vert \nabla u \Vert_{\infty} \le M_0$, $ \Vert v \Vert_{\infty} + \Vert \nabla v \Vert_{\infty} \le M_0$. Since by construction $ 2 \delta +\vep < \tau_0 $ and $ \rho < \eps_0$, we can apply Theorem \ref{th:wcpstrip} to conclude that  $ u \le u_{\bar{\lambda}+\vep} $ on $\mathcal{S}_{(2\delta + \vep,\rho)}. $ This clearly contradicts the definition of $\mathcal{S}_{(2\delta + \vep,\rho)}$. Hence $\mathcal{S}_{(2\delta + \vep,\rho)} = \emptyset$, which concludes the proof.
\begin{flushright}
$\square$
\end{flushright}

\section{Proof of Theorem \ref {mainthmnonpositive} and Theorem \ref {liouvillenextgenerationtris}
}\label{sect.cortjlhjdligld}

\noindent \textbf{Proof of Theorem \ref {mainthmnonpositive} .}\\
 \noindent  Since we assumed that $b(u)\geq 0$ then  $-\Delta_p u\leq f(u)$ so that
 Theorem 1.7 in \cite{FMS} applies and gives that
 actually $$0<u\leq z$$
 in $\mathbb{R}^N_+$. Note that, once it is proved that $0<u\leq z$, then the strong maximum principle (see \cite{PSB}) applies
 and gives that actually $0<u<z$. It follows furthermore that $u$ is strictly bounded away from $z$ in $\Sigma_\lambda$ for any $\lambda>0$.  In fact, if this is not the case, arguing as in the proof of Theorem \ref{trittofritto} (see case a)) we could easily construct a limiting profile $u_\infty$ with $0<u_\infty\leq z$, touching $z$ at some point. This is not possible again by the strong maximum principle \cite{PSB}.
 This is enough to repeat the proof  of Theorem \ref{mainthm} and get the thesis.\\

\noindent \textbf{Proof of  Theorem \ref {liouvillenextgenerationtris}.}\\
\noindent If $u$ is not identically zero, then it is strictly positive by the strong maximum principle (see \cite{PSB,V}). Therefore
 $u$ is monotone increasing w.r.t. the $y$-direction by Theorem \ref{mainthm} and
the proof of $b)$ and of $c)$ follows by \cite{MP,SZ} exactly in the same way as $b)$ and $c)$ in Theorem 1.6 of \cite{FMS}.\\

\noindent To prove $a)$ let $N=2$ and denote by $(x,y)$ a point in the plane. Define
\begin{equation}\begin{split}
& w(x)
:=\lim_{y\rightarrow \infty} u(x,y)\\
\end{split}\end{equation}
We see that (see e.g. the proof of Theorem 8.3 in \cite{FMS} for details) $w:\mathbb{R}\rightarrow \mathbb{R}$ is  non-negative and bounded with
\[
-(|w'|^{p-2}w')'=f(w)\geq 0\,.
\]
A simple O.D.E analysis shows that $w$ is constant and, by the assumptions on $f$, it follows that necessarily $w=0$ that also implies $u=0$ and the thesis. \\

To prove the non-existence result when $ f(s) >0$ for $ s \ge0$, we first consider the case $ N=2$. By the above argument (which uses only the assumption $ f(s) >0$ for $ s > 0$) we infer that $u=0$ which contradicts $ f(0) >0$. Thus, there are no non-negative solutions. The same argument can be employed to treat the case $ N \ge 3$. Indeed, in this case the assumption $ f(0) >0$ implies that $f(s) \geq \lambda s^\frac{(N-1)(p-1)}{N-1-p}$ in $[0,\delta]$, for some $\lambda,\delta>0$, yielding $ u=0$. Again contradicting $ f(0) >0$.
\begin{flushright}
$\square$
\end{flushright}


\bigskip

\end{document}